%% file: main.tex
\documentclass[
    a4paper,
    10pt,
    oneside,
    onecolumn,
    3p,
    number,
    sort&compress
]{elsarticle}

\input{tex/packages}

\input{tex/commands}

\input{tex/glossaries}

\begin{document}

\begin{frontmatter}
\title{On the Dynamics of Linear Finite Dynamical Systems Over Galois Rings}

\author[1]{Jonas Kantic\corref{cor1}}
\ead{jonas.kantic@tum.de}

\author[2]{Claudio Qureshi}
\author[3]{Daniel Panario}
\author[1]{Fabian C. Legl}

\cortext[cor1]{Corresponding author}

\affiliation[1]{
    organization={Technical University of Munich, School of Computation, Information and Technology},
    postcode={80333},
    city={Munich},
    country={Germany}
}
\affiliation[2]{
    organization={Instituto de Matemática y Estadística Rafael Laguardia, Facultad de Ingeniería, Universidad de la República},
    postcode={11300},
    city={Montevido},
    country={Uruguay}
}
\affiliation[3]{
    organization={School of Mathematics and Statistics, Carleton University},
    postcode={K1S 5B6},
    city={Ottawa},
    country={Canada}
}
\input{sections/0_abstract}

\end{frontmatter}

\input{sections/1_introduction}
\input{sections/2_preliminaries}
\input{sections/3_generalized_hensel}
\input{sections/4_cycle_lengths}
\input{sections/5_transient_length}
\input{sections/6_conclusion}

\section*{Acknowledgment}
Jonas Kantic was partially funded by the Bavarian Ministry of Economic Affairs, Regional Development and Energy in the context of
the Bavarian Collaborative Research Program (BayVFP), funding line
Digitization, funding area Information, and Communication Technology (Grant no. DIK-2104-0055// DIK0320/01).

Claudio Qureshi was partially supported by the STIC-AmSud project MOV\_CO\_2024\_9\_1500328 and by CSIC UdelaR.

D. Panario was partially funded by the Natural Sciences and Engineering Research Council of Canada (NSERC), reference number RGPIN-2024-05341.

\bibliographystyle{elsarticle-num}
\bibliography{Promotion}
\end{document}

%% file: tex/packages.tex
\usepackage{lmodern}
\usepackage[utf8]{inputenc}
\usepackage[T1]{fontenc}
\usepackage{microtype}
\usepackage{graphicx}
\usepackage{amsmath, amsfonts, amssymb, amsthm}
\usepackage{array}
\usepackage{stackrel}
\usepackage{bm}
\usepackage{mathrsfs}
\usepackage{mathtools}
\usepackage{xcolor}
\usepackage{hyperref}
\usepackage[capitalize]{cleveref}
\usepackage[ruled]{algorithm}
\usepackage{algpseudocode}
\usepackage{mleftright}
\usepackage[xindy]{glossaries}
\usepackage{tikz-cd}
\glsdisablehyper

%% file: tex/commands.tex
\newcommand{\intnum}{\mathbb{Z}}
\newcommand{\natnum}{\mathbb{N}}

\newcommand{\set}[1]{\mathcal{#1}}

\newcommand{\mat}[1]{\bm{\MakeUppercase{#1}}}

\renewcommand{\vec}[1]{\bm{\MakeLowercase{#1}}}

\DeclareMathOperator{\lcm}{lcm}
\newcommand{\isomorphic}{\cong}
\newcommand{\order}{\mathcal{O}}

\DeclareMathOperator{\rad}{rad}
\DeclareMathOperator{\per}{per}
\DeclareMathOperator{\pper}{pper}
\newcommand{\ibar}[2][3]{{}\mkern#1mu\overline{\mkern-#1mu#2}}
\newcommand{\bmat}[1]{\ibar{\mat{#1}}}

\DeclareMathOperator{\sys}{}

\newcommand{\erel}[1]{\vcenter{\hbox{\(\stackrel{#1}{\sim}\)}}}

\newcommand{\field}{\mathbb{F}}

\DeclareMathOperator{\gring}{GR}
\newcommand{\pring}[1]{#1\mleft[ x \mright]}
\newcommand{\qring}[2]{#1/#2}

\newcommand{\intmodring}{\intnum}
\newcommand{\ideal}[1]{\langle #1 \rangle}

\newcommand{\csetsym}{\mathcal{L}}
\newcommand{\cset}{\csetsym}
\newcommand{\csetp}{\csetsym'}
\newcommand{\fgraph}{\mathcal{G}}

\algrenewcommand\algorithmicrequire{\textbf{Input:}}
\algrenewcommand\algorithmicensure{\textbf{Output:}}
\algnewcommand\algorithmicforeach{\textbf{for each}}
\algdef{S}[FOR]{ForEach}[1]{\algorithmicforeach\ #1\ \algorithmicdo}

\newtheorem{lemma}{Lemma}
\newtheorem{proposition}{Proposition}
\newtheorem{theorem}{Theorem}
\newtheorem{corollary}{Corollary}

\theoremstyle{definition}
\newtheorem{definition}{Definition}
\newdefinition{example}{Example}

%% file: tex/glossaries.tex
\newacronym{lcm}{lcm}{Least Common Multiple}
\newacronym[longplural={Cellular Automata}]{ca}{CA}{Cellular Automaton}
\newacronym{gf}{GF}{Galois Field}
\newacronym{gr}{GR}{Galois Ring}
\newacronym{rc}{RC}{Reservoir Computing}
\newacronym{fds}{FDS}{Finite Dynamical System}
\newacronym{lfds}{LFDS}{Linear Finite Dynamical System}
\newacronym{clfds}{CLFDS}{Cyclic Linear Finite Dynamical System}
\newacronym{pbcfds}{PBCFDS}{Primary Bijective Circulant Finite Dynamical System}
\newacronym{adam}{ADAM}{Analysis of Dynamic Algebraic Models}
\newacronym{snf}{SNF}{Smith Normal Form}
\newacronym{ufd}{UFD}{Unique Factorization Domain}
\newacronym{pid}{PID}{Principal Ideal Domain}

%% file: sections/0_abstract.tex
\begin{abstract}
    Linear finite dynamical systems play an important role, for example, in coding theory and simulations.
    Methods for analyzing such systems are often restricted to cases in which the system
    is defined over a field 
    or lack practicability to effectively analyze the system's dynamical behavior.
    However, when analyzing and prototyping finite dynamical systems, it is often desirable to quickly obtain basic information such as
    the length of cycles and transients that appear in its dynamics, which is reflected in the structure of the connected components
    of the corresponding functional graphs.
    In this paper, we extend the analysis of the dynamics of linear finite dynamical systems that act over cyclic modules to Galois rings.
    Furthermore, we propose algorithms for computing the length of the cycles and the height of the trees that make up
    their functional graphs.
\end{abstract}

%% file: sections/1_introduction.tex
\section{Introduction}
\label{sec:introduction}

\Glspl{lfds} represent a fundamental entity in various fields of research.
They are employed, for example, in coding theory \cite{Gadouleau2018}, in the design and analysis of linear feedback shift registers \cite{Chang2017},
in the characterization of involutions relevant to cryptographic applications \cite{Panario2018},
or in simulations of regulatory genetic networks \cite{Bollman2007} and multi-agent networks over finite fields \cite{Pasqualetti2014,Meng2020}.
In this paper, we focus on a subclass of \glspl{lfds} that act on cyclic modules.
A recent application of such \glspl{lfds} uses finite uniform linear \glspl{ca} as dynamical reservoirs within the reservoir computing framework for machine learning-based time-series processing \cite{Kantic2024}.

When investigating \glspl{lfds}, it is essential to identify the system's dynamical behavior.
A common approach is to decompose their dynamics into cycles and trees.
Information about the set of cycle lengths and the height of the trees provides crucial insights into their dynamics.
Extensive academic research on the dynamics of \glspl{lfds} has been conducted, for example, in \cite{Toledo2005,Deng2015,Wei2016,Panario2018,Qureshi2019,Rohde2024}.

In practical applications, it is often desirable to quickly obtain information about the system's general behavior, i.e., identifying the cycle and transient lengths present in its functional graph.
However, existing approaches either target systems defined over finite fields and lack generalization toward finite commutative rings, or they rely on the complex computation of minimal polynomials and do not leverage the specific structure of the target module, resulting in algorithms with high runtime complexity that hinder the analysis and undermine the practical applicability of theoretical results.

In this work, we analyze the structure of the connected components in the functional graphs of \glspl{lfds} that act on cyclic modules over \glspl{gr}.
This includes the set of cycle lengths and the maximal number of steps after which the systems stabilize in a cycle,
which corresponds to the height of the trees in their functional graphs.
Furthermore, we propose efficient algorithms that compute the set of cycle lengths and the height of the trees of such \glspl{lfds}.

The rest of this work is structured as follows.
In \cref{sec:preliminaries}, we provide an overview of the background on relevant concepts in abstract algebra and the theory of \glspl{fds}.
Subsequently, in \cref{sec:hensel} we discuss the decomposition of \glspl{lfds} on cyclic modules over \glspl{gr} into products of a nilpotent and a bijective subsystem
based on an extension of Hensel's lifting lemma.
This is followed in \cref{sec:cycles} by an analysis of the cycle lengths that appear in the dynamics of such \glspl{lfds},
which extends to the discussion of our proposed algorithm and its runtime complexity.
After that, we propose and discuss an algorithm to compute the height of the trees in \cref{sec:height}.
Finally, conclusions are drawn in \cref{sec:conclusion}.

%% file: sections/2_preliminaries.tex
\section{Preliminaries and Related Work}
\label{sec:preliminaries}
This section introduces the necessary concepts and establishes the notation required for subsequent discussions.
First, we provide a brief overview of the essential background in abstract algebra in \cref{subsec:abstract-algebra}.
Next, we introduce relevant background on the theory of \glspl{fds} in \cref{subsec:fds}, including \glspl{lfds}.
In \cref{subsec:related-work}, we provide an overview of related work.

\subsection{Background in Abstract Algebra}
\label{subsec:abstract-algebra}

Let \(p\) be a prime number and fix some monic polynomial \(h \in \intnum[t]\) of degree \(d = \deg(h)\) such that \(h\) is irreducible modulo \(p\).
If \(\ideal{h}\) is the principal ideal generated by \(h\), then
the ring of truncated polynomials with degree less than \(d\) over \(\intnum\) can be constructed as a quotient ring
\[R \coloneqq \qring{\intnum[t]}{\ideal{h}}.\]


The following lemma shows that the ring \(R\), as well as certain quotient rings of the form \(\qring{\pring{R}}{\ideal{m}}\), \(m\in\pring{R}\), are torsion-free as \(\mathbb{Z}\)-module (i.e. \(nr=0\) with \(n\in \mathbb{Z}, n\neq 0\) implies \(r=0\)). Throughout the paper, all rings are assumed to be commutative with unity.

\begin{lemma}
\label{lem:Z-free-modules}
Let \(S\) be a ring and let \( g \in \pring{S} \) be a monic polynomial. If \(S\) is torsion-free, then \( \qring{S[x]}{\ideal{g}}\) is also torsion-free. In particular, \(R = \qring{\mathbb{Z}[t]}{\ideal{h}}\) is torsion-free.
\end{lemma}

\begin{proof}
Let $a\in \mathbb{Z}$ and $r \in S[x]$. We have to prove that if $a r(x) \equiv 0 \pmod{g(x)}$ with $a\neq 0$ then $r(x)\equiv 0 \pmod{g(x)}$. The congruence above can be written as $a r(x)=g(x)k(x)$ for some $k\in S[x]$. If $a \nmid k(x)$ we can write $k(x)= ax^{m+1} k_0(x)+ b x^m + k_1(x)$ with $a \nmid b$ and $\deg(k_1)<m$. Then, all the coefficients of $bx^m g(x) + k_1(x)g(x)=a(r(x)-x^{m+1}k_0(x))$ are divisible by \(a\). However, since $g$ is monic and  $\deg(k_1)<m$, the leading coefficient of $bx^m g(x) + k_1(x)g(x)$ is \(b\), which is not divisible by $a$. This yields a contradiction. In particular, since $h\in \mathbb{Z}[t]$ is monic and $\mathbb{Z}$ is torsion-free (as $\mathbb{Z}$-module) we have that $R$ is free as a $\mathbb{Z}$-module.
\end{proof}


We denote the ring of integers modulo \(p\) as \(\intmodring_p\).
Recall that \(\intmodring_p\) is isomorphic to the finite field of characteristic \(p\), denoted by \(\field_p\).
An extension field of order \(q=p^d\) can be constructed as
\begin{equation*}
    \field_q \isomorphic \qring{\intnum_p[t]}{\ideal{h(t)}} \isomorphic \qring{\intnum[t]}{\ideal{h(t),p}} \isomorphic \qring{R}{\ideal{p}}.
\end{equation*}

\Glspl{gr} are generalizations of finite fields.
For any \(e\in\intnum^+\), a Galois ring of characteristic \(p^e\) and extension degree \(d\) can be defined as
\begin{equation*}
    \gring(p^e,d) \isomorphic \qring{\intnum_{p^e}[t]}{\ideal{h(t)}} \isomorphic \qring{\intnum[t]}{\ideal{h(t), p^e}} \isomorphic \qring{R}{\ideal{p^e}}.
\end{equation*}
Clearly, if \(e=1\), then we have \(\gring(p,d) \isomorphic \field_q\).
We call \(\field_q\) the residue field of \(\gring(p^e,d)\).
To simplify notation, for fixed \(p\) and \(d\), we define \(R_e\coloneqq\gring(p^e,d)\).
In the remainder of this work, we follow the convention of using uppercase letters for polynomials over \(R_1=\field_q\), and lowercase letters to denote polynomials over general rings.

A well-known theorem states that quotient rings decompose into direct products.
Recall that two ideals \(I\) and \(J\) of a ring \(S\) are coprime, if \(I + J = S\), that is, there exists \(i \in I\) and \(j \in J\) such that \(i + j = 1\).
If one of them is principal, for example \(I = fS\) for some \(f\in S\), then for simplicity say that \(f\) is coprime to \(J\) (instead of stating that \(fS\) is coprime to \(J\)).
\begin{theorem}
    \label{the:crt}
    \emph{(Chinese Remainder Theorem)} Let \(S\) be a commutative ring and \(I_1, I_2, \dots, I_n\) be mutually coprime ideals of \(S\).
    If \(I=I_1 I_2 \cdots I_n\), then \(\qring{S}{I}\) is isomorphic to a direct product as follows
    \[\qring{S}{I} \isomorphic \qring{S}{I_1} \times \qring{S}{I_2} \times \dots \times \qring{S}{I_n}.\]
\end{theorem}
The proof of \cref{the:crt} is elementary, see \cite[Part II, Theorem 17]{Dummit2004}, for example.
The following lemmas are useful tools when working with prime-power moduli. The first allows us to lift congruences to higher prime powers, while the second allows us to lift B{\'e}zout identities to higher prime powers.
\begin{lemma}
    \label{lem:power-p}
    Let \(S\) be a commutative ring and \(a,b \in S\). If \(e \ge 1\) and \(a \equiv b \pmod{p^e}\), then \(a^p \equiv b^p \pmod{p^{e+1}}\).
\end{lemma}

\begin{proof}
See, for example, \cite[p.42]{Ireland1990} for a proof in the ring of integers. The argument extends straightforwardly to any commutative ring \(S\).
\end{proof}

\begin{lemma}
	\label{lem:lifting-Bezout}
Let \(S\) be a commutative ring, \(p\) a prime number, and \(i \geq 1\). Let \(a,b \in S\). If there exist \(x,y \in S\) such that \(xa - yb \equiv 1 \pmod{p^i}\), then there exist \(x',y' \in S\) such that \(x'a - y'b \equiv 1 \pmod{p^{i+1}}\). In particular, if \(a\) and \(b\) are coprime modulo \(p\), they are coprime modulo \(p^e\), \(\forall e\geq 1\).
\end{lemma}

\begin{proof}
    Let \(z\in S\) such that \(xa=yb+1+zp^i\).
    We have that \( (xa)^p = (yb+1+zp^i)^p \equiv (yb+1)^p +p\cdot (yb+1)^{p-1}\cdot zp^i \pmod{p^{i+1}} \).
    Expanding \( (yb+1)^p \), we can write it in the form \(y'b+1\) for some \(y'\in S\).
    Thus, \(x'a-y'b\equiv 1 \pmod{p^{i+1}}\) where \(x'=x\cdot (xa)^{p-1}\)
\end{proof}

\paragraph{Canonical Homomorphism}
Based on the structure of \glspl{gr} defined by prime-power modular arithmetic, homomorphisms between them naturally emerge by reducing the power of the prime.
Similarly, there exists a canonical homomorphism from \(R\) to \(R_e\) by modular reduction with \(p^e\).
\begin{definition}
    \emph{(Canonical Homomorphism)}
    For any non-negative integer  \(e\), we define the map
        \[
            \pi_e:
                \begin{cases}
                    R &\rightarrow R_e\\
                    z &\mapsto z \pmod{p^e}.
                \end{cases}
        \]
    The map \(\pi_e\) is commonly referred to as the canonical (or natural) homomorphism from \(R\) to \(R_e\).
\end{definition}
We note that \(\pi_e\) naturally extends to polynomial rings and modules over \(R\), and we use the same symbol \(\pi_e\) to denote the corresponding homomorphisms in these cases.
Moreover, we use the shorthand notation \(\bar{z} \coloneqq \pi_1(z) \in \field_q\) to denote the homomorphism between \(R\) and its residue field of characteristic \(p\).
This notation also extends to polynomials.

\paragraph{Factorization of Polynomials}
A well-known result is that factorizations of polynomials modulo an element \(m\) in a commutative ring with identity can be lifted to factorizations modulo \(m^k\) for any \(k > 0\).
\begin{lemma}
    \label{lem:hensel}
    \emph{(Hensel's Lifting Lemma)}
    \cite[Theorem 15.12 and Algorithm 15.10]{Gathen1999}
    Let \(S\) be a commutative ring with identity, \(m \in S\), and \(f,g,h,s,t \in \pring{S}\), where \(h\) is monic and the leading coefficient of \(f\) is not a zero-divisor
    modulo \(m\), satisfying \(\deg(f) = n = \deg(g) + \deg(h)\), \(\deg(s) < \deg(h)\), and \(\deg(t) < \deg(g)\),
    such that
    \[f \equiv g\cdot h \pmod{m} \text{ and } s\cdot g + t\cdot h \equiv 1 \pmod{m}.\]
    Then for any \(k \in \intnum^+\), there exist polynomials \(g',h',s',t' \in \pring{S}\), \(h'\) monic,
    satisfying \(g' \equiv g \pmod{m}\), \(h' \equiv h \pmod{m}\), \(s' \equiv s \pmod{m}\), \(t' \equiv t \pmod{m}\), \(\deg(g') = \deg(g)\), \(\deg(h') = \deg(h)\),
    \(\deg(s') < \deg(h')\) and \(\deg(t') < \deg(g')\),
    such that
    \[f \equiv g'\cdot h' \pmod{m^k} \text{ and } s'\cdot g' + t' \cdot h' \equiv 1 \pmod{m^k}.\]
\end{lemma}
The proof of \cref{lem:hensel} is elementary, and we skip the details here.
For details and applications, see, for example, \cite[pp. 433-472]{Gathen1999}.

In general, polynomial rings over \Glspl{gr} are not unique factorization domains because they contain zero divisors.
Nevertheless, by \cref{lem:hensel}, certain polynomials defined over \glspl{gr} admit unique factorization.
Recall that \(R_e:= R/\ideal{p^e}\) is a finite local ring whose maximal ideal is \(\ideal{p}\), the set of all multiples of \(p\). In particular, $R_1=\mathbb{F}_q$ is the finite field with $q=p^d$ elements.
Moreover, the units of \(R_e\) are precisely the elements of \(R_e \setminus \ideal{p}\). A non-zero polynomial \(f(x) = \sum_{i=0}^r a_i x^i \in \pring{R_e}\) is called {\it regular} if it is not a zero divisor. The following facts hold, and their proofs can be found in \cite[Theorem 2.1]{Salagean2005}.
\begin{itemize}
    \item \(f\) is a zero-divisor if and only if \(a_i \in \ideal{p}\) for \(i = 0, \dots, r\);
    \item \(f\) is regular if and only if \(a_i \in R_e\setminus\ideal{p}\) for some \(i\), \(0 \le i \le r\);
    \item \(f\) is a unit if and only if \(a_0 \in R_e\setminus\ideal{p}\) and \(a_i \in \ideal{p}\) for \(i=1,\cdots,r\).
\end{itemize}
If \(f\) is regular, then there are polynomials \(u, f_0 \in \pring{R_e}\) with \(u\) being a unit and \(f_0\) monic, such that:
\begin{equation}
    \label{eq:regular-poly}
    f = u\cdot f_0.
\end{equation}

Consider $F \in \pring{\field_q}$ such that \(F \equiv f \pmod{p}\). If \(f\) is both regular and non-unit, then \(f\) is primary if and only if \(F= u\cdot G^m\) where $u\in \field_q^*$ and $G \in \pring{\field_q}$ is an irreducible polynomial \cite[Theorem 2.2]{Salagean2005}. In particular, if \(f\) is a non-constant monic polynomial, then \(f\) is primary if and only if \(F\) is a power of an irreducible polynomial.

Any non-constant monic polynomial \(f \in \pring{R_e}\) can be factored uniquely into regular primary and mutually coprime polynomials \cite[Theorem 2.3]{Salagean2005}.
Let \(F = F_1^{m_1} \cdots F_r^{m_r}\), where \(F_i \in \pring{\field_q}\) are monic irreducible polynomial that are pairwise coprime, and \(m_i \in \intnum^+\) for \(i=1,\dots,r\).
Using Hensel lifting, this factorization over \(\field_q = R_1\) can be lifted to $R_e$, yielding a factorization \(f=f_1 \cdots f_r\), where \(F_i^{m_i} \equiv f_i \pmod{p}\). Hence, the polynomials \(f_i\) are primary. Moreover, the polynomials \(f_i\) are mutually coprime in the sense that \(f_i\pring{R_e} + f_j\pring{R_e} = \pring{R_e}\) for any \(1 \le i< j \le r\). We observe, however, the individual factors \(f_i\) are not necessarily irreducible.


\subsection{Finite Dynamical Systems}
\label{subsec:fds}

\Glspl{fds} are typically defined as tuples \(\sys(\set{X}, \Gamma)\), where \(\set{X}\) is a finite set forming the state space,
and \(\Gamma: \set{X} \rightarrow \set{X}\) is the system function that defines how the system transitions from one state to the next \cite{Toledo2005}.
Hence, \(\Gamma\) defines the overall dynamical behavior of the system.
Two \glspl{fds} \((\set{X}, \Gamma_1)\) and \((\set{Y}, \Gamma_2)\) are equivalent, denoted by \((\set{X}, \Gamma_1) \isomorphic (\set{Y}, \Gamma_2)\), if
there is a bijection \(\varphi\colon\set{X}\rightarrow\set{Y}\) such that \(\Gamma_2 = \varphi\circ\Gamma_1\circ\varphi^{-1}\).

Starting at any state \(x \in \set{X}\), the repeated application of \(\Gamma\) to \(x\) defines a path through the state space, called the orbit of \(x\).
We denote the \(k\)-times repeated application of \(\Gamma\) to \(x\) as \(\Gamma^{k}(x)\).
Since the state space is finite, any orbit eventually stabilizes in a periodic pattern.
This leads us to the definition of transients and cycles in the orbits of an \gls{fds}.
\begin{definition}
    \emph{(Transients and Cycles)}
    For any \(x \in \set{X}\), let \(\tau\) be the least non-negative integer and \(\ell\) be the least positive integer satisfying 
    \(\Gamma^{\tau+\ell}(x) = \Gamma^{\tau}(x).\)
    Then the set \(\{x, \Gamma(x), \dots, \Gamma^{\tau-1}(x)\}\) is the transient in the orbit of \(x\),
    and \(\pper(\Gamma, x) \coloneqq \tau\) the transient length (or preperiod).
    Moreover, the set \(\{\Gamma^{\tau}(x), \Gamma^{\tau+1}(x), \dots, \Gamma^{\tau+\ell-1}(x)\}\) is the cycle in the orbit of \(x\),
    and \(\per(\Gamma, x) \coloneqq \ell\) the cycle length (or period).
\end{definition}
Fundamental operations in the context of \glspl{fds} are sums and products \cite[Definition 2]{Toledo2005}.
\begin{definition}
    \label{def:notation-sum-product}
    \emph{(Sum and Product of \glspl{fds})}
    Let \(\sys(\set{X},\Gamma_1)\) and \(\sys(\set{Y},\Gamma_2)\) be two \glspl{fds}.
    The sum of the two systems is given by
    \(\sys(\set{X}, \Gamma_1) \oplus \sys{\set{Y}, \Gamma_2} = \sys(\set{X} \lor \set{Y}, \Gamma_1 \lor \Gamma_2),\)
    where \(\set{X} \lor \set{Y}\) denotes the disjoint union of \(\set{X}\) and \(\set{Y}\),
    and \(\Gamma_1 \lor \Gamma_2: \set{X} \lor \set{Y} \rightarrow \set{X} \lor \set{Y}\) is given by
    \[(\Gamma_1 \lor \Gamma_2)(z) \coloneqq
        \begin{cases}
            \Gamma_1(z) &\text{if } z \in \set{X},\\
            \Gamma_2(z) &\text{else}.
        \end{cases}\]
    In addition, if \(n\in\intnum^+\), then we write
    \[n \cdot \sys(\set{X},\Gamma) \coloneqq \underbrace{\sys(\set{X},\Gamma) \oplus \dots \oplus \sys(\set{X}, \Gamma)}_{n\text{ summands}}.\]
    The sum of a series of systems is denoted by \(\bigoplus\).
    
    The (tensor) product of the two \glspl{fds} is given by
    \(\sys(\set{X},\Gamma_1) \otimes \sys(\set{Y},\Gamma_2) = \sys(\set{X}\times\set{Y}, \Gamma_1 \times \Gamma_2),\)
    where \(\set{X}\times\set{Y}\) denotes the Cartesian product of \(\set{X}\) and \(\set{Y}\),
    and \((\Gamma_1 \times\Gamma_2)(x,y) = (\Gamma_1(x),\Gamma_2(y))\) for any \(x\in\set{X}, y\in\set{Y}\).
    The product of a series of systems is denoted by \(\bigotimes\).
\end{definition}
An important observation is that \(\oplus\) and \(\otimes\) obey associativity and distributivity, and that \(\oplus\) is commutative.
However, it is clear that \(\otimes\) is not commutative since the Cartesian product already is not.
Nevertheless, there is a straightforward isomorphism between \(\set{X} \times \set{Y}\) and \(\set{Y}\times\set{X}\) such that
\(\sys(\set{X}, \Gamma_1) \otimes \sys(\set{Y},\Gamma_2) \isomorphic \sys(\set{Y},\Gamma_2) \otimes \sys(\set{X},\Gamma_1).\)
Therefore, the product of systems is unique up to isomorphism.

The dynamics of an \gls{fds} \(\sys(\set{X},\Gamma)\) can be represented as a graph \cite{Toledo2005} in which the vertices are the elements in \(\set{X}\),
and for any pair of states \(x_1, x_2 \in \set{X}\), there is a directed edge from \(x_1\) to \(x_2\) if and only if \(\Gamma(x_1) = x_2\).
More precisely, we define such graphs as follows.
\begin{definition}
    \emph{(Functional Graph)}
    To any \gls{fds} \(\sys(\set{X},\Gamma)\) we associate a graph \(\fgraph(\Gamma) = (\set{V},\set{E})\), where \(\set{V} = \set{X}\) is the set of vertices,
    and \(\set{E} = \mleft\{ (x,\Gamma(x)) : x \in \set{V}\mright\}\) is the set of directed edges.
    We call \(\fgraph(\Gamma)\) the functional graph of \(\sys(\set{X},\Gamma)\).
\end{definition}
The functional graph of an \gls{fds} is a complete characterization of the structure of its state space and dynamical behavior, and can thus be used to
identify the corresponding \gls{fds}.
Equivalent \glspl{fds} correspond to isomorphic functional graphs, and
the sums and products of systems translate to sums and products of their functional graphs \cite{Toledo2005,Qureshi2019}.
Therefore, we reuse the same notation as introduced in \cref{def:notation-sum-product} for functional graphs.
There are two basic building blocks in the functional graphs of \glspl{fds}, namely cycles and trees.
Cycles in graphs correspond to the cycles in the orbits of \glspl{fds}.
\begin{definition}
    \emph{(Cycles)}
    For any \(\ell\in\intnum^+\), a cycle is a directed and weakly connected graph with \(\ell\) vertices and \(\ell\) edges such that the indegree and outdegree of every vertex are both equal to one.
    We denote such a cycle by \(\mathcal{C}_\ell\), where \(\ell\) is the length of the cycle.
\end{definition}
Cycles of length \(\ell=1\) are called fixed points.
In what follows, \(\gcd\) and \(\lcm\) denote the greatest common divisor and the least common multiple, respectively.
\begin{proposition}
    \label{pro:prod-cycles}
    \cite{Elspas1959}, \cite[Proposition 3.1]{Toledo2005}
    The product of \((a\cdot\mathcal{C}_i)\) by \((b\cdot\mathcal{C}_j)\) is given by
    \[(a\cdot\mathcal{C}_i) \otimes (b\cdot\mathcal{C}_j) \isomorphic (a \cdot b \cdot \gcd(i,j)) \cdot \mathcal{C}_{\lcm(i,j)}.\]
\end{proposition}
On the other hand, trees correspond to the transient parts in the orbits of \glspl{fds}.
\begin{definition}
    \emph{(Trees)}
    A tree is a directed, weakly connected, acyclic graph with exactly one vertex with no successors, referred to as the root.
    Vertices with no predecessors are leaves.
    A tree is directed from its leaves to the root.
    From each vertex in the tree, there exists exactly one path to the root.
    The height of a tree is the maximum number of edges on the path from a leaf to the root.
    A tree of height zero is trivial and consists of the root and no other vertices.
\end{definition}

It has been shown that the functional graph of an \gls{fds} is the sum of weakly connected components,
where each component consists of a cycle whose vertices are roots of (possibly trivial) trees \cite{Toledo2005}.
A notable form of such components is given by a single tree rooted at a fixed point, which we refer to as a looped tree.
In the remainder of this work, we denote a looped tree as \(\mathring{\mathcal{T}}\), and the corresponding ordinary tree (without the loop at the root) as \(\mathcal{T}\).
The product of a looped tree of height \(i\) by a looped tree of height \(j\) results in a looped tree of height \(\max(i,j)\) \cite[Proposition 3.2]{Toledo2005}.
Moreover, the product of a looped tree \(\mathring{\mathcal{T}}\) by a cycle \(C_\ell\) results in a cycle of length \(\ell\), of which each vertex
is the root of a tree isomorphic to \(\mathcal{T}\) (when ignoring the edges of the cycle) \cite[Proposition 3.4]{Toledo2005}.

\begin{definition}
    \emph{(System Height)} The height of an \gls{fds} is the maximum height of its trees.
\end{definition}

An important concept in the analysis of a system's dynamics is the set of cycle lengths that appear in the functional graph of an \gls{fds} \cite{Deng2015}.
\begin{definition}
    \emph{(Set of Cycle Lengths)}
    For any \gls{fds} \(\sys(\set{X}, \Gamma)\), we define \(\cset(\Gamma)\) to be the set of the lengths of cycles that appear in \(\fgraph(\Gamma)\).
    More formally, \(\cset(\Gamma) = \bigl\{ \per(\Gamma, x) : x \in \set{X}\bigr\}.\)
    Furthermore, if \(\Gamma\) is a ring of characteristic \(p\), then we define
    \(\csetp(\Gamma) = \{\ell : \ell \in \cset(\Gamma), p \nmid \ell\}.\)
\end{definition}

\subsubsection{\texorpdfstring{\Acrlongpl{lfds}}{Linear Finite Dynamical Systems}}
\label{sssec:lfds}
Recall that \(R = \qring{\intnum[t]}{\ideal{h}}\), where \(h \in \intnum[t]\) is a fixed polynomial of degree \(d\geq 1\) that is irreducible modulo a fixed prime \(p\), \(R_e = \qring{R}{\ideal{p^e}}\) and \(\pi_e:R \to R_e\) is the canonical projection map.
In this work, we consider the case \(\mathcal{X}=R_e^n\), \(n\in\intnum^+\), where the system function is linear over \(R_e^n\).
We refer to such systems as \glspl{lfds}.
Let \(\mat{A} = (a_{i,j})_{1 \le i,j \le n}\), \(a_{i,j} \in R\).
Moreover, let us define \(\pi_e(\mat{A}) \coloneqq (\pi_e(a_{i,j}))_{1 \le i,j \le n}\), and \(\bmat{A} \coloneqq \pi_1(\mat{A})\).
Then the system function of an \(\gls{lfds}\) \((R_e^n, \Gamma_e(\mat{A}))\) can be defined in terms of \(\mat{A}\):
\[
    \Gamma_e(\mat{A}):
        \begin{cases}
            R_e^n &\rightarrow R_e^n\\
            \vec{x} &\mapsto \pi_e(\mat{A})\vec{x}.
        \end{cases}
\]
We refer to \(\mat{A}\) as the system matrix.

As we show in this section, the cycle lengths that appear in the dynamics of \glspl{lfds} can be obtained via the order of certain associated polynomials.
\begin{definition}
    \label{def:order}
    \emph{(Generalized Order)}
    Let \(S\) be a finite commutative ring with identity and \(f, g \in \pring{S}\) be two polynomials with \(f\pring{S} + g\pring{S} = \pring{S}\).
    Then we define the polynomial order \(\order(f,g)\) to be the smallest positive integer \(k\) satisfying \(f^k \equiv 1 \pmod{g}\). 
    In case \(S = R\), we additionally define \(\order_e(f,g) \coloneqq \order(\pi_e(f),\pi_e(g))\).
\end{definition}
\begin{lemma}
    \label{lem:order-lcm-fq}
    Let \(f,g \in \pring{R}\) such that \(\gcd(\bar{f}, \bar{g}) = 1\).
    Furthermore, let \(\bar{g}(x) = G_1(x) \cdots G_r(x)\) such that the \(G_i \in \pring{\field_q}\), \(1 \le i \le r\), are mutually coprime.
    Then \(\order_1(f,g) = \lcm(\order(\bar{f},G_1), \allowbreak \dots, \allowbreak \order(\bar{f},G_r))\).
\end{lemma}
The proof of \cref{lem:order-lcm-fq} is elementary, and we skip the details here.
\begin{theorem}
    \label{the:order-x-pow}
    \cite[Theorem 3.8]{Lidl1996}
    Let \(g \in \pring{R}\) such that \(\bar{g}\) is irreducible and \(\gcd(x, \bar{g}) = 1\).
    Then for any \(k\in\intnum^+\), \(\order_1(x,g^k) = p^\tau \cdot \order_1(x,g)\), where \(\tau\) is the smallest integer satisfying \(p^\tau \ge k\).
\end{theorem}

\paragraph{\texorpdfstring{\Glspl{lfds}}{LFDSs} over Galois fields}
The theory of \glspl{lfds} has been discussed in depth in \cite{Elspas1959,Toledo2005}.
In \cite{Elspas1959}, the author considers linear sequential networks defined over prime fields as one form of \glspl{lfds}.
It is shown that the cyclic behavior of an \gls{lfds} over a prime field can be analyzed in terms of a factorization of the characteristic
polynomial of \(\bmat{A}\) into elementary divisors.
In \cite{Toledo2005}, the results of \cite{Elspas1959} have been generalized for \glspl{lfds} over extension fields and embedded in a framework for
analyzing the dynamics of general \glspl{lfds} in terms of a decomposition into sums and products of looped trees and cycles.
In \cite{Stevens1999}, the author uses the same idea to decompose linear \glspl{ca}, another form of \glspl{lfds}, into cycles and trees.

A brief outline of the typical process for analyzing the dynamics of general \glspl{lfds} \(\sys(\field_q^n, \Gamma_1(\mat{A}))\) is as follows \cite{Elspas1959,Stevens1999,Toledo2005,Deng2015,Wei2016}.
First, the invariant factors of \(\bmat{A}\) are computed based on the \gls{snf} of \(\bmat{A}-x\mat{I}\) over \(\field_q\), where \(\mat{I}\) is the identity matrix.
Via the \gls{snf} we obtain a diagonal matrix \(\operatorname{diag}(1,\dots,1,\Lambda_1, \dots, \Lambda_r)\), where \(\Lambda_i \in \pring{\field_q}\), \(i=1,\dots, r\),
such that \(\Lambda_j(x) \mid \Lambda_k(x)\) for each \(j \le k\).
The polynomials \(\Lambda_i\) form the invariant factors of \(\bmat{A}\).
A well-known result in abstract algebra is the structure theorem for finitely generated modules over a \gls{pid},
stating that every such module is isomorphic to a sum of cyclic modules generated by its invariant factors.
Hence, if \(\mathcal{V}(\mat{A})\) denotes \(\field_q^n\) viewed as an \(\pring{\field_q}\)-module where \(x\) acts as multiplication by \(\bmat{A}\), then we obtain
\begin{equation}
    \label{eq:module-isomorphism}
    \mathcal{V}(\mat{A}) \isomorphic \qring{\pring{\field_q}}{\ideal{\Lambda_1}} \oplus \cdots \oplus \qring{\pring{\field_q}}{\ideal{\Lambda_r}}.
\end{equation}
The representation in \cref{eq:module-isomorphism} is directly connected to the rational canonical form of \(\bmat{A}\).
Let \(n_i = \deg(\Lambda_i)\) and \(\bmat{A}_i\) be the companion matrix of \(\Lambda_i(x)\).
Then there exists a decomposition of \(\sys(\field_q^n, \Gamma_1(\mat{A}))\) given as
\begin{equation}
    \label{eq:inv-fac-subsys}
    \sys(\field_q^n, \Gamma_1(\mat{A})) \isomorphic \bigotimes_{i=1}^r \sys(\field_q^{n_i}, \Gamma_1(\mat{A}_i)),
\end{equation}
where \(n=n_1 + \dots + n_r\).

Each subsystem \(\sys(\field_q^{n_i}, \Gamma_1(\mat{A}_i))\) can be further decomposed into smaller subsystems.
This can be accomplished via a factorization of its respective invariant factor as follows
\begin{equation}
    \label{eq:inv-fac-fac}
    \Lambda_i(x) = x^{s_i} \prod_{k=1}^{j_i} A_{i,k}(x)^{u_{i,k}},
\end{equation}
where \(s_i \in \natnum\), \(u_{i,k} \in \intnum^+\), and \(A_{i,k}(x)\) are irreducible and mutually coprime polynomials with \(A_{i,k}(0) \ne 0\). 
The factors of \(\Lambda_i(x)\) as given in \cref{eq:inv-fac-fac} are referred to as the elementary divisors of \(\bmat{A}\) \cite{Stevens1999,Toledo2005}.
Defining \(\Lambda_{i,0}(x) \coloneqq x^{s_i}\) and \(\Lambda_{i,k}(x) \coloneqq A_{i,k}(x)^{u_{i,k}}\), then by \cref{the:crt} we obtain
\begin{equation}
    \label{eq:crt-primary}
    \qring{\pring{\field_q}}{\ideal{\Lambda_i}} \isomorphic \qring{\pring{\field_q}}{\ideal{\Lambda_{i,0}}} \oplus \qring{\pring{\field_q}}{\ideal{\Lambda_{i,1}}} \oplus \cdots \oplus \qring{\pring{\field_q}}{\ideal{\Lambda_{i,j_i}}}.
\end{equation}
Consequently, if \(b_{i,k} = \deg(\Lambda_{i,k})\), \(\bmat{A}_{i,0}\) is the companion matrix of \(\Lambda_{i,0}(x)\), and \(\bmat{A}_{i,k}\) is the companion matrix of \(\Lambda_{i,k}(x)\), then
a subsystem \(\sys(\field_q^{n_i}, \Gamma_1(\mat{A}_i))\) can be written as a product of smaller subsystems
\begin{equation}
    \label{eq:subsys-deco}
    \sys(\field_q^{n_i}, \Gamma_1(\mat{A}_i)) \isomorphic \sys(\field_q^{s_i}, \Gamma_1(\mat{A}_{i,0})) \otimes \Biggl(\bigotimes_{k=1}^{j_i} \sys(\field_q^{b_{i,k}}, \Gamma_1(\mat{A}_{i,k}))\Biggr),
\end{equation}
where \(n_{i} = s_{i} + b_{i,1} + \dots + b_{i,j_i}\).
Moreover, \(\bmat{A}_{i,0}\) is nilpotent of index \(s_i\),
and \(\bmat{A}_{i,k}\) are invertible matrices with \(\det\bigl(\bmat{A}_{i,k}\bigr) \ne 0\).

It has been shown in \cite{Toledo2005} that \(\fgraph(\Gamma_1(\mat{A}_{i,0}))\) is a looped tree of height \(s_i\) that is rooted at the zero fixed point.
On the other hand, each of the graphs \(\fgraph(\Gamma_1(\mat{A}_{i,k}))\) represents a sum (i.e., disjoint union) of cycles, which can be described in a closed form
\cite{Elspas1959,Toledo2005}
\begin{equation}
    \label{eq:closed-form}
    \fgraph(\Gamma_1(\mat{A}_{i,k})) \isomorphic \mathcal{C}_1 + \sum\limits_{j=1}^{u_{i,k}} \mleft(\frac{q^{dj} - q^{d(j-1)}}{\ell_j}\mright) \cdot \mathcal{C}_{\ell_j},
\end{equation}
where \(d=\deg(A_{i,k})\), \(\ell_j=\order(x,A_{i,k}^j)\), and the first summand \(\mathcal{C}_1\) corresponds to the zero fixed point.

As can be seen, the analysis of any \gls{lfds} can be reduced to the case in which the characteristic and minimal polynomial of the system matrix coincide.
The dynamics of the overall \gls{lfds} can then be obtained by multiplying the trees and cycles of the subsystems together using \cref{pro:prod-cycles}
and \cite[Proposition 3.2]{Toledo2005}.
Another observation is that we can extract and group the nilpotent and the bijective factors, giving rise to the fact that any \gls{lfds} can be decomposed into a
product of a nilpotent subsystem by a bijective subsystem and is thus the product of a looped tree (rooted at the zero fixed point) by a sum of cycles,
as summarized in \cite[Theorem 1]{Toledo2005}.
By reordering the terms in \cref{eq:inv-fac-subsys,eq:subsys-deco}, we obtain
\begin{equation}
    \label{eq:deco-nil-bij}
    \sys(\field_q^n, \Gamma_1(\mat{A})) \isomorphic \underbrace{\mleft(\bigotimes_{i=1}^{r} \sys(\field_q^{s_i}, \Gamma_1(\mat{A}_{i,0}))\mright)}_{\text{nilpotent}} \otimes
    \underbrace{\mleft(\bigotimes_{i=1}^{r}\bigotimes_{k=1}^{j_i} \sys(\field_q^{b_{i,k}}, \Gamma_1(\mat{A}_{i,k}))\mright)}_{\text{bijective}}.
\end{equation}
In \cite{Rohde2024}, the decomposition into a product of a nilpotent subsystem by a bijective subsystem has been generalized for \glspl{lfds} defined over finite commutative rings, including \glspl{gr}.
An upper bound for the height of \glspl{lfds} has been proposed in \cite{Lindenberg2018}.

\begin{theorem}
    \label{the:lindenberg}
    \cite[Adapted from Theorem B]{Lindenberg2018} Let \(\sys(R_e^n, \Gamma_e(\mat{A}))\) be an \gls{lfds} and \(h\) be its height.
    If \(h_1\) is the height of \(\sys(\field_q^n, \Gamma_1(\mat{A}))\), then \(h \le h_1\cdot e\).
\end{theorem}

Based on the general decomposition of the functional graph of \glspl{lfds} as discussed above,
the height of the nilpotent subsystem and the set of cycle lengths of the bijective subsystem are
substantial pieces of information about the dynamics of the systems.
They allow us to identify the structure of the connected components in the corresponding functional graphs.
For \glspl{lfds} whose characteristic and minimal polynomial equal a power of an irreducible polynomial, the set of cycle lengths follows a certain pattern, as the following result shows. 
\begin{lemma}
    \label{lem:cset-base}
    \cite[Lemma 4.1]{Deng2015}
    Let \(G\in\pring{\field_q}\) be irreducible with \(G(0)\ne0\), \(m\in\natnum\), and \(\sys(\field_q^n, \Gamma_1(\mat{A}))\) be an \gls{lfds} such that both
    the characteristic and minimal polynomial of \(\bmat{A}\) are equal to \(G(x)^m\).
    Furthermore, let \(\ell = \order(x,G)\) and \(\tau\) be the smallest integer satisfying \(p^\tau \ge m\).
    Then \(\cset(\Gamma_1(\mat{A})) = \{1, \ell, p\cdot\ell, \dots, p^\tau \cdot \ell\}\).
\end{lemma}

\paragraph{\texorpdfstring{\Glspl{lfds}}{LFDSs} over Galois rings}
Let us now assume the general case in which \(e > 1\).
When switching from \(\field_q\) to \(R_e\), the analysis of the dynamics of an \gls{lfds} becomes more difficult, because \(R_e\) is not a unique factorization domain.
In \cite{Deng2015}, however, the author describes the cyclic structure of \(\fgraph(\Gamma_e(\mat{A}))\), and how cycles in the homomorphic
functional graph over \(\field_q\) lift to cycles in \(\fgraph(\Gamma_e(\mat{A}))\).
The author especially elaborates on the structure of \(\cset(\Gamma_e(\mat{A}))\).
\begin{corollary}
    \label{cor:cset-structure}
    \cite[Corollary 5.1]{Deng2015}
    \(\cset(\Gamma_e(\mat{A})) \subseteq \big\{p^j\ell : \ell \in \cset(\Gamma_1(\mat{A})),\allowbreak 0 \le j < e \big\}\).
\end{corollary}
\begin{definition}
    \label{def:omega}
    \emph{(Cycle Length Factor)}
    For any \(\ell \in \csetp(\Gamma_e(\mat{A}))\), we define
    \[\omega(\Gamma_e(\mat{A}), \ell) = \max\{i : p^i\ell \in \cset(\Gamma_e(\mat{A}))\}.\]
\end{definition}
\begin{theorem}
    \label{the:omega}
    \cite[Theorem 5.1]{Deng2015}
    Let \(\ell \in \csetp(\Gamma_e(\mat{A}))\). Then \(\omega(\Gamma_{e+1}(\mat{A}),\ell) = \omega(\Gamma_e(\mat{A}),\ell)\)
    or \(\omega(\Gamma_{e+1}(\mat{A}),\ell)=\omega(\Gamma_e(\mat{A}),\ell) + 1\) for any \(e \ge 1\).
\end{theorem}

\subsubsection{\texorpdfstring{\Glspl{lfds}}{LFDSs} on Cyclic \texorpdfstring{\(\pring{\field_q}\)-Modules}{Modules Over Finite Fields}}
\label{sssec:clfds}
As we have seen in \cref{eq:inv-fac-subsys}, any \gls{lfds} defined over \(\field_q\) can be written as a product of linear subsystems
of the form \(\sys(\field_q^{n_i}, \Gamma_1(\mat{A}_i))\), where \(\field_q^{n_i}\) is isomorphic to \(\qring{\pring{\field_q}}{\ideal{\Lambda_i}}\),
viewed as a cyclic \(\pring{\field_q}\)-module.
Based on this isomorphism, we can construct a polynomial representation of \glspl{lfds} acting on cyclic modules.

Let \(S\) be a commutative ring with identity, \(f\in S\), and \(I\) be an ideal of \(S\).
We define the multiplication-by-\(f\)-map \(\Gamma(f,I)\) as follows
\[
    \Gamma(f,I):
        \begin{cases}
            \qring{S}{I} &\rightarrow \qring{S}{I} \\
            g &\mapsto fg
        \end{cases},
\]
considering \(\qring{S}{I}\) as an \(S\)-module.
As it is customary, we identify an element in \(S\) with the principal ideal it generates.
Following this practice, if \(I=\ideal{m}\) denotes a principal ideal generated by some \(m\in S\), we identify \(\Gamma(f,I)\) with \(\Gamma(f,m)\).
In the case that \(S\) is an \(R\)-algebra, we additionally define \(\Gamma_e(f, I) \coloneqq \Gamma(\pi_e(f),I)\).
Considering the module structures given in \cref{eq:module-isomorphism,eq:crt-primary}, and setting \(S=\pring{\field_q}\), and \(f(x) = x\), we obtain the isomorphisms
\begin{align}
    \label{eq:system-isomorphism-if}
    \sys(\field_q^n, \Gamma_1(\mat{A}_i)) &\isomorphic \sys(\qring{\pring{\field_q}}{\ideal{\Lambda_i}}, \Gamma_1(f,\Lambda_i)), \\
    \sys(\field_q^{s_i}, \Gamma_1(\mat{A}_{i,0})) &\isomorphic \sys(\qring{\pring{\field_q}}{\ideal{\Lambda_{i,0}}}, \Gamma_1(f,\Lambda_{i,0})), \text{ and}\\
    \sys(\field_q^{b_{i,k}}, \Gamma_1(\mat{A}_{i,k})) &\isomorphic \sys(\qring{\pring{\field_q}}{\ideal{\Lambda_{i,k}}}, \Gamma_1(f,\Lambda_{i,k})).
\end{align}

Let us recall some standard definitions.
The radical of \(I\) is the ideal \(\rad(I) = \{r \in S : r^n \in I \text{ for some } n \in \intnum^+\}\).
Furthermore, if \(S\) is a \gls{ufd} and \(f=\prod_{i=1}^mf_i^{\alpha_i}\) is the factorization of \(f\) into irreducibles (with each \(\alpha_i > 0\)),
then \(g\in\rad(fS)\) if and only if \(\rad(f)\coloneqq\prod_{i=1}^m f_i \mid g\).
In other words, \(\rad(fS)\) is precisely the principal ideal generated by \(\rad(f)\).
\begin{definition}
    \label{def:f-deco}
    \emph{(\(f\)-decomposition)}
    Let \(f \in \set{S}\) and \(I_1, I_2\) be ideals of \(\set{S}\).
    An \(f\)-decomposition of an ideal \(I\) of \(\set{S}\) is a decomposition of the form \(I = I_1 \cdot I_2\) such that \(fS\) is coprime with \(I_1\) (i.e. \(I_1 + fS = S\), or equivalently \(a+bf=1\) for some \(a\in I_1\) and \(b\in S\)) and \(f\in \rad(I_2)\).
\end{definition}
The fact that \(f\in\rad(I_2)\) implies the existence of an integer \(N \ge 0\) such that \(f^N\in I_2\).
Moreover, since \(f\) and \(I_1\) are coprime, we can find \(a \in S\) and \(b\in I_1\) such that \(af + b = 1\), which, when raised to the \(N\)th power and regrouping elements in \(I\), gives \(a^Nf^N + b' = 1\) with \(f^N \in I_2\) and \(b' \in I_1\).
Applying \cref{the:crt}, we then obtain an isomorphism \(\qring{S}{I} \isomorphic \qring{S}{I_1} \times \qring{S}{I_2}\), which leads to a decomposition of the map \(\Gamma(f,I) = \Gamma(f,I_1) \times \Gamma(f,I_2)\).
It is easy to verify that \(\Gamma(f,I_1)\) is bijective, while the map \(\Gamma(f,I_2)\) is nilpotent.

In \cite{Qureshi2019}, the authors investigate the dynamics of linear maps \(\Gamma(f,I)\) in the case where \(S\) is a Dedekind domain.
They achieve a complete characterization by decomposing \(I\) into a set of useful submodules for analyzing linear systems acting on cyclic modules.
Consider the case \(S = \pring{\field_q}\), noting that \(S\) is an Euclidean domain, and let \(F, M \in \pring{\field_q}\) be nonzero.
In this case, the \(F\)-decomposition of \(M\) in \(\pring{\field_q}\) is of the form \(M = M_1 \cdot M_2\) and can be obtained by setting
\(M_1 \in \pring{\field_q}\) equal to the greatest monic divisor of \(M\) that is coprime with \(F\), and letting \(M_2 \in \pring{\field_q}\) be defined by \(M = M_1 \cdot M_2\).
It is easy to verify that \(F\in\rad(M_2)\).
Such an \(f\)-decomposition allows us to construct the functional graph of \glspl{lfds} acting on \(\qring{\pring{\field_q}}{\ideal{M}}\).
For example, the construction of \(\fgraph(\Gamma(f,I))\) for \(S=\pring{\field_q}\) and \(I=\ideal{x^n-1}\) is thoroughly discussed in \cite{Panario2018,Qureshi2019}.

\begin{theorem}
    \label{the:qureshi}
    \cite[Adapted from Theorem 3.6]{Qureshi2019}
    Let \(F, M\in\pring{\field_q}\) be nonzero polynomials with \(1 \le \deg(F) < \deg(M)\).
    Suppose that \(M = M_1 \cdot M_2\) is the \(F\)-decomposition of \(M\) in \(\pring{\field_q}\) such that \(F\) is coprime with \(M_1\) and \(F\in\rad(M_2)\).
    Furthermore, let \(\mathring{\mathcal{T}}_{M_2(F)} = \fgraph(\Gamma(F, M_2))\).
    Then the following isomorphism holds:
    \begin{equation}
        \label{eq:qureshi}
    \fgraph(\Gamma(F,M)) \cong \bigoplus_{G \mid M_1} \frac{\mleft\vert \mleft(\qring{\pring{\field_q}}{\ideal{G}}\mright)^\times\mright\vert}{\order(F,G)} \cdot \mleft(\mathring{\mathcal{T}}_{M_2(F)} \otimes \mathcal{C}_{\order(F,G)} \mright).
    \end{equation}
\end{theorem}
The proof of \Cref{the:qureshi} is covered by \cite[Theorem 3.6]{Qureshi2019} for the special case of linear systems over \(\field_q\).
We note that by the above discussion, \(\mathring{\mathcal{T}}_{M_2(F)}\) is the nilpotent part of \(\fgraph(\Gamma(F,I))\) and represents a looped tree rooted at the zero fixed point.
By distributive law, we can rewrite \cref{eq:qureshi} as
\begin{equation}
    \label{eq:qureshi-deco}
    \fgraph(\Gamma(F,M)) \isomorphic \mathring{\mathcal{T}}_{M_2(F)} \otimes \mleft(\bigoplus\limits_{G \mid M_1}\frac{\mleft\vert \mleft(\qring{\pring{\field_q}}{\ideal{G}}\mright)^\times\mright\vert}{\order(F,G)} \cdot \mathcal{C}_{\order(F,G)}\mright),
\end{equation}
highlighting the fact that such a polynomial-based analysis similarly describes a decomposition of a \glspl{lfds} into the product of a tree by a sum of cycles,
and thus represents a decomposition into a product of a nilpotent and a bijective subsystem.

A method to compute \(\mathcal{T}_{M_2(F)}\)
is described in \cite[Definition 2.3]{Qureshi2015}.
In particular, the following result shows how to determine the height of \(\mathcal{T}_{M_2(F)}\).
\begin{proposition}
    \label{pro:height-field}
    \cite[Adapted from Proposition 2.4]{Qureshi2015}
    Let \(F, M\in\pring{\field_q}\) be nonzero polynomials with \(1 \le \deg(F) < \deg(M)\).
    Suppose that \(M = M_1 \cdot M_2\) is the \(F\)-decomposition of \(M\) in \(\pring{\field_q}\) such that \(\ideal{F}\) is coprime with \(\ideal{M_1}\) and \(F\in\rad(M_2)\).
    Furthermore, let \(F(x) = \prod_{i=1}^s G_i(x)^{k_i}\), where \(G_i \in \pring{\field_q}\) are irreducible and mutually coprime polynomials.
    If \(\kappa_{G_i}(z)\) is the exponent of \(G_i(x)\) in \(z(x)\), then the height \(h\) of the tree \(\mathcal{T}_{M_2(F)}\) associated with \(M_2\) and \(F\) is given by
    \(h = \max\{\lceil \kappa_{G_i}(M_2) / \kappa_{G_i}(F) \rceil : G_i \mid F\}.\)
\end{proposition}
\begin{proof}
    The height of a tree associated with \(M_2\) and \(F\) is given by \cite[Proposition 2.14]{Qureshi2015}.
    The closed form of its height is a straightforward generalization of \cite[Proposition 2.4]{Qureshi2015}.
\end{proof}

In \cref{sec:hensel}, we show how to lift \(f\)-decompositions of \glspl{lfds} from finite fields to Galois rings based on a novel extension of Hensel's lifting lemma.

\subsection{Related Work}
\label{subsec:related-work}
Fundamental work in the theory and analysis of \glspl{lfds} over finite fields was done in \cite{Elspas1959}, which was extended and generalized in \cite{Toledo2005}.
Several papers build upon this work and generalize towards finite rings \cite{Xu2009,Mendivil2012,Deng2015,Wei2016,Rohde2024}.
The authors in \cite{Qureshi2015,Panario2018,Qureshi2019} developed a polynomial-based approach for analyzing \glspl{lfds}.

Apart from research on general \glspl{lfds}, numerous studies target more specific forms of \glspl{lfds} or specific application domains.
One branch considers finite linear \glspl{ca} with periodic boundary conditions, corresponding to \glspl{lfds} acting on cyclic modules \cite{Martin1984,Guan1986,Jen1988,Jen1988a,Stevens1999}.
Many of these studies yield results that align with other, more general results presented in papers on \glspl{lfds}, for example, \cite{Stevens1999}.
Another notable field is systems biology, which focuses not only on \glspl{lfds} but also on (nonlinear) polynomial finite dynamical systems \cite{Laubenbacher2004,ColonReyes2006a,Jarrah2008}.
In \cite{Hinkelmann2011}, a software tool called \gls{adam} has been proposed and applied to gene regulatory networks, analyzes the cyclic behavior of polynomial \glspl{fds} using Gröbner bases.

%% file: sections/3_generalized_hensel.tex
\section{\texorpdfstring{\Glspl{lfds}}{LFDSs} on Cyclic \texorpdfstring{\(\pring{R_e}\)-Modules}{Modules Over Galois Rings}}
\label{sec:hensel}
Henceforth, we consider a polynomial \(m\in\pring{R}\) of degree \(\deg(m) = n > 1\) that generates a principal ideal \(I = m\pring{R}\), 
and a monic polynomial \(f \in \pring{R}\) with \(1 \le \deg(f) < n\).
To further simplify notation, we
identify a map \(\Gamma(f,I)\) (or \(\Gamma_e(f,I)\)) with the \gls{lfds} \(\sys(\qring{\pring{R}}{I}, \Gamma(f,I))\) (or \(\sys(\qring{\pring{R_e}}{I}, \Gamma_e(f,I))\)) it defines.
We have that \(\pi_e(f)\) is regular.

For \(e=1\), we already discussed the construction of the \(f\)-decomposition of \(\ideal{m}\) in \(\pring{R_1}\isomorphic\pring{\field_q}\) in \cref{sssec:clfds}.
Our goal is to lift \(f\)-decompositions modulo \(p^i\) to \(f\)-decompositions modulo \(p^{i+1}\) for any \(i \ge 1\).
The following result shows that it is possible to lift an \(f\)-decomposition in \(\pring{R_1}\) to any Galois ring \(\pring{R_e}\) for \(e \ge 1\).

\begin{theorem}
    \label{the:general-hensel}
    Let \(f,m \in \pring{R}\) be polynomials with \(1 \le \deg(f) < \deg(m)\).
    Suppose that there are polynomials \(m_1, m_2, h_1, \alpha, \beta \in \pring{R}\) and \(N \in \intnum^+\) such that
    \begin{equation*}
        \begin{cases}
            m \equiv m_1 m_2 \pmod{p^i},\\
            \alpha f + \beta m_1 \equiv 1 \pmod{p^i},\\
            f^N\equiv h_1 m_2 \pmod{p^i}.
        \end{cases}
    \end{equation*}
    Furthermore, define \(h_2 = h_1^p m_2^{p-1}\).
    Then there are polynomials \(m_1',m_2',h_2',\alpha',\beta' \in \pring{R}\) with \(m_1'\equiv m_1 \pmod{p^i}\), \({m_2' \equiv m_2 \pmod{p^i}}\), \({h_2' \equiv h_2 \pmod{p^i}}\),
    \({\alpha' \equiv \alpha \pmod{p^i}}\), \({\beta' \equiv \beta \pmod{p^i}}\) such that
    \begin{equation*}
        \begin{cases}
            m \equiv m_1' m_2' \pmod{p^{i+1}},\\
            \alpha' f + \beta' m_1' \equiv 1 \pmod{p^{i+1}},\\
            f^{Np} \equiv h_2' m_2' \pmod{p^{i+1}}.
        \end{cases}
    \end{equation*}
\end{theorem}

\begin{proof}
    Let \(a,b,c \in \pring{R}\) such that \(m = m_1 m_2 + a p^i\), \(\alpha f + \beta m_1 = 1 + b p^i\), and \(f^N = h_1 m_2 + c p^i\).
    Moreover, let \(C_j^N = \binom{N}{j}\).
    By the binomial theorem, we have that
    \[(\alpha f + \beta m_1)^N = (1 + b p^i)^N \quad \mbox{if and only if} \quad \alpha^N f^N + \beta_0 m_1 = 1 + b_0 p^i,\]
    where \(\beta_0 = \sum_{j=0}^{N-1} C_j^N \beta^{N-j} (\alpha f)^j m_1^{N-1-j}\) and \(b_0 = \sum_{j=1}^N C_j^N b^j p^{i(j-1)}\).

    Since \(\alpha^N f^N = \alpha^N h_1 m_2 + \alpha^N c p^i\), we have that
    \begin{equation}
        \label{eq:hensel-1}
        \alpha^N h_1 m_2 + \beta_0 m_1 = 1 + (b_0 - \alpha^N c) p^i.
    \end{equation}
    We also have the following congruences
    \[f^{Np}=(f^N)^p = (h_1 m_2 + c p^i)^p \equiv h_1^p m_2^p + p h_1^{p-1} m_2^{p-1} c p^i \equiv h_2 m_2 \pmod{p^{i+1}}.\]
     Write \(m_1' = m_1 + u_1 p^i\), \(m_2' = m_2 + u_2 p^i\), \(h_2' = h_2 + v p^i\) for some \(u_1, u_2, v \in \pring{R}\) .
    Then the following congruences hold
\begin{align}
\label{eq:hensel-2}
m_1' m_2' 
&\equiv m - a p^i + (m_2 u_1 + m_1 u_2) p^i 
\equiv m + (m_2 u_1 + m_1 u_2 - a) p^i 
\pmod{p^{i+1}}, \\
\label{eq:hensel-3}
h_2' m_2' 
&\equiv f^{Np} + (h_2 u_2 + m_2 v) p^i 
\equiv f^{Np} + m_2 (h_1^p m_2^{p-2} u_2 + v) 
\pmod{p^{i+1}}.
\end{align}

    By \cref{eq:hensel-1,eq:hensel-2}, setting \(u_1 = \alpha^N h_1 a\) and \(u_2 = \beta_0 a\), we get \(m_1' m_2' \equiv m \pmod{p^{i+1}}\).
    By \cref{eq:hensel-3}, setting \(v= -h_1^p m_2^{p-2} u_2\), we obtain \(h_2' m_2' \equiv f^{Np} \pmod{p^{i+1}}\).
    We have that \(\alpha f + \beta m_1' = \alpha f + \beta (m_1 + u_1 p^i) = (\alpha f + \beta m_1) + \beta u_1 p^i = 1 + (b + \beta u_1) p^i = 1 + u p^i\),
    where \(u \coloneq b + \beta u_1\).
    Setting \(\alpha' = (1-u p^i)\alpha\) and \(\beta' = (1 - u p^i)\beta\), we obtain
    \[\alpha' f + \beta' m_1' = (1 - u p^i)(\alpha f + \beta m_1') = (1 - u p^i)(1 + u p^i) \equiv 1 \pmod{p^{i+1}}.\]
\end{proof}

We can use \cref{the:general-hensel} to decompose \glspl{lfds} over cyclic \(\pring{R_e}\)-modules for any \(e \ge 1\), as the following result demonstrates.
\begin{corollary}
    \label{cor:deco-nil-bij}
    Let \(f,m \in \pring{R}\) be polynomials with \(1 \le \deg(f) < \deg(m)\).
    Furthermore, let \(M_1, M_2 \in \pring{R_1}\) such that \(m \equiv M_1 \cdot M_2 \pmod{p}\) is an \(f\)-decomposition of \(m\pring{R_1}\),
    that is, \(f\pring{R_1}\) is coprime with \(M_1\), and \(f \in \rad(M_2)\).
    Then for any \(e \ge 1\), there exist \(m_1, m_2 \in \pring{R_e}\) with \(\bar{m}_1 = M_1\), \(\bar{m}_2 = M_2\), and \(m \equiv m_1 \cdot m_2 \pmod{p^e}\) such that
    the \gls{lfds} \(\sys(\qring{\pring{R_e}}{\ideal{m}}, \Gamma_e(f,m))\) admits a decomposition into a product of a nilpotent and a bijective subsystem as follows
    \[\sys(\qring{\pring{R_e}}{\ideal{m}},\Gamma_e(f,m)) \isomorphic \underbrace{\sys(\qring{\pring{R_e}}{\ideal{m_2}},\Gamma_e(f,m_2))}_{\text{nilpotent}} \otimes \underbrace{\sys(\qring{\pring{R_e}}{\ideal{m_1}},\Gamma_e(f,m_1))}_{\text{bijective}}.\]
\end{corollary}
\begin{proof}
    Because \(m \equiv M_1 \cdot M_2 \pmod{p}\) is an \(f\)-decomposition of \(m\pring{R_1}\), we have that \(\gcd(M_1, \allowbreak M_2) = 1\).
    By repeated application of \cref{the:general-hensel}, there exist \(b, s \in \pring{R_e}\) with \(\bar{b} = M_1\), \(\bar{s} = M_2\),
    and \(m \equiv b \cdot s \pmod{p^e}\) such that
    \(f\pring{R_e}\) and \(b\) are coprime and \(f \in \rad(s)\).
    Hence, there exist some \(N \in \intnum^+\) and \(\alpha, \beta \in \pring{R_e}\) such that \(f^N \in s\pring{R_e}\) and \(\alpha f + \beta b \equiv 1 \pmod{p^e}\).
    It follows that \(b\) and \(s\) form an \(f\)-decomposition of \(m\pring{R_e}\), such that \(b\pring{R_e}\) and \(s\pring{R_e}\) are coprime.
    By \cref{the:crt}, we have the isomorphism
    \[\qring{\pring{R_e}}{\ideal{m}} \isomorphic \qring{\pring{R_e}}{\ideal{s}} \times \qring{\pring{R_e}}{\ideal{b}}.\]
    This isomorphism allows us to decompose the map \(\Gamma_e(f,m)\), yielding \(\Gamma_e(f,m) \isomorphic \Gamma_e(f,s) \times \Gamma_e(f,b).\)
    It is easy to check that \(\Gamma_e(f,b)\) is bijective and \(\Gamma_e(f,s)\) is nilpotent.
    Setting \(m_1=b\) and \(m_2=s\), we obtain the claimed result.
\end{proof}


In order to analyze the cyclic behavior of \(\Gamma_e(f,m_1)\), our approach is to decompose the bijective system into smaller bijective systems. When considering systems over \(R_1 = \field_q\), this can be accomplished based on a factorization of \(\bar{m}_1\) and then, such decomposition can be lifted to \(R_e\) for any \(e \ge 1\) using \cref{lem:hensel}.

\begin{corollary}
\label{cor:deco-bij}
Suppose that \(\bar{m}_1(x)=G_1(x)\cdots G_r(x),
\) where \(G_i\in \pring{R_1}\), \(1\le i\le r\), are pairwise coprime polynomials.
Then the bijective \glspl{lfds} defined by \(\Gamma_1(f,m_1)\) decomposes as
\[
\Gamma_1(f,m_1) \isomorphic \bigotimes_{i=1}^r \Gamma_1(f,G_i).
\]

Moreover, for each \(e>1\), there exist pairwise coprime polynomials
\(g_1,\dots,g_r\in \pring{R_e}\) such that \(\bar{g}_i=G_i\) for \(1\le i\le r\) and 
\(m_1=g_1 g_2 \cdots g_r\) .
In this case, the bijective \gls{lfds} defined by \(\Gamma_e(f,m_1)\) decomposes as
\[
\Gamma_e(f,m_1) \isomorphic \bigotimes_{i=1}^r \Gamma_e(f,g_i).
\]
\end{corollary}

\begin{proof}
For the first part, note that the Chinese Remainder Theorem yields the isomorphism
\[
\qring{\pring{R_1}}{\ideal{m_1}} \cong 
\qring{\pring{R_1}}{\ideal{G_1}} \times \cdots \times 
\qring{\pring{R_1}}{\ideal{G_r}}.
\]
As a consequence, the map \(\Gamma_1(f,m_1)\) decomposes as
\[
\Gamma_1(f,m_1) =
\Gamma_1(f,G_1) \times \cdots \times \Gamma_1(f,G_r).
\]
Since \(\Gamma_1(f,m_1)\) is bijective, each map \(\Gamma_1(f,G_i)\) is also bijective.

For the second part, the existence of pairwise coprime polynomials 
\(g_i \in \pring{R_e}\) is guaranteed by applying Hensel's lemma to the factorization over the base field. 
Again, by the Chinese Remainder Theorem we obtain
\[
\qring{\pring{R_e}}{\ideal{m_1}} \cong 
\qring{\pring{R_e}}{\ideal{g_1}} \times \cdots \times 
\qring{\pring{R_e}}{\ideal{g_r}}.
\]
Via this isomorphism, we obtain a decomposition
\[
\Gamma_e(f,m_1) =
\Gamma_e(f,g_1) \times \cdots \times \Gamma_e(f,g_r),
\]
which yields the claimed decomposition of the \gls{lfds}.
\end{proof}

Based on the above results, each bijective \gls{lfds} defined by \(\Gamma_e(f,m_1)\) can be written as a product of smaller bijective systems \(\Gamma_e(f,g_i)\).


It is clear that the functional graph \(\fgraph(\Gamma_e(f,m_2))\) of the nilpotent subsystem in \cref{cor:deco-nil-bij} consists of a looped tree \(\mathring{\mathcal{T}}\) rooted at the zero fixed point such that
\(\mathcal{T}\) is isomorphic to the trees appearing in the connected components of \(\fgraph(\Gamma_e(f,m))\).
On the other hand, \(\fgraph(\Gamma_e(f,m_1))\) is a disjoint union of cycles that coincide with those in \(\fgraph(\Gamma_e(f,m))\).

In the following sections, we construct \(\cset(\Gamma_e(f,m))\) based on a factorization of \(m_1\) (see \cref{sec:cycles}),
and compute the height of the tree defined by \(m_2\) (see \cref{sec:height}).
In both cases, the orbit of \(1\) plays a crucial role, which is shown in the following result.

\begin{theorem}
    \label{the:orbit-1}
    Let \(f,m \in \pring{R}\) be polynomials with \(1 \le \deg(f) < \deg(m)\).
    Then the following holds:
    \begin{enumerate}
        \item For any \(\ell \in \cset(\Gamma_e(f,m))\), it holds that \(\ell\) divides \(\per(\Gamma_e(f,m),1)\).
        \item The height \(h\) of the system \(\Gamma_e(f,m)\) is given by \(h = \pper(\Gamma_e(f,m), 1)\).
    \end{enumerate}
\end{theorem}
\begin{proof}
    Let \(\tau_1 = \pper(\Gamma_e(f,m), 1)\) and \(\ell_1 = \per(\Gamma_e(f,m), 1)\).
    Then we have
    \[f^{\tau_1 + \ell_1} \equiv f^{\tau_1} \pmod{(m, p^e)}.\]
    Multiplying both sides by some \(g \in \pring{R}\), we obtain
    \[f^{\tau_1 + \ell_1} \cdot g \equiv f^{\tau_1} \cdot g \pmod{(m, p^e)}.\]
    If \(\tau = \pper(\Gamma_e(f,m), g)\) and \(\ell = \per(\Gamma_e(f,m), g)\), then by the principle of minimality of \(\tau\) and \(\ell\) it follows that
    \begin{enumerate}
        \item \(\ell \mid \ell_1\), and
        \item \(\tau \le \tau_1\), leading to \(h = \tau_1\).
    \end{enumerate}
\end{proof}

%% file: sections/4_cycle_lengths.tex
\section{Bijective \texorpdfstring{\glspl{lfds}}{LFDSs} on Cyclic \texorpdfstring{\(\pring{R_e}\)-Modules}{Modules Over Galois Rings}}
\label{sec:cycles}

Fix \(f, m \in\pring{R}\) with $m$ monic, \(1\leq \deg(f)<\deg(m)\), and $e\in \mathbb{Z}^+$. We return to the construction of \(\cset(\Gamma_e(f,m))\).
Let \(m \equiv m_1 \cdot m_2 \pmod{p^e}\) be the \(f\)-decomposition of \(m\) in \(\pring{R_e}\) (identifying \(f\) with its image in \(\pring{R_e}\)), that is, $m_1,m_2 \in \pring{R_e}$ such that
\(f\) is coprime to \(m_1\) and \(f \in \rad(m_2)\) as in \cref{cor:deco-nil-bij}.
Throughout this section, we focus on the bijective subsystem defined by \(\Gamma_e(f,m_1)\), because by the above discussion, we have that
\(\cset(\Gamma_e(f,m)) = \cset(\Gamma_e(f,m_1))\).

First, let us generalize \cref{the:order-x-pow}.
\begin{theorem}
    \label{the:order-pow}
    Let \(f,g \in \pring{R}\) be polynomials such that \(\bar{g}\) is irreducible and \(\gcd(\bar{f}, \bar{g})=1\), and \(t\in\intnum^+\).
    Furthermore, let \(s\) be the greatest integer satisfying \(\bar{g}^s \mid \bar{f}^{\order(\bar{f},\bar{g})} - 1\).
    Then \(\order(\bar{f},\bar{g}^t) = p^k \cdot \order(\bar{f},\bar{g})\), where \(k\) is the smallest integer satisfying \(t/s \le p^k\).
\end{theorem}
\begin{proof}
    Let \(F=\bar{f}\), \(G=\bar{g}\), and \(\ell = \order(F,G)\). 
    First, consider the case \(k=0\).
    Then  \(t\le s\), and
    \[F^\ell\equiv1\pmod{G^t},\]
    resulting in \(\order(F,G^t)\mid\ell\).
    On the other hand, from
    \[F^{\order(F,G^t)}\equiv 1 \pmod{G^t}\]
    it follows that \[F^{\order(F,G^t)}\equiv1\pmod{G},\]
    leading to \(\ell \mid \order(F,G^t)\).
    Since \(\order(F,G^t) \mid \ell\) and \(\ell \mid \order(F,G^t)\), it directly follows that \(\order(F,G^t) = \ell\).

    Now, consider the case \(k\ge1\) and let \(H \in \pring{\field_q}\) with \(G \nmid H\) satisfying
    \begin{equation}
        \label{eq:order-pow-1}
        F^\ell = G^s \cdot H + 1.
    \end{equation}
    Suppose that \(\order(F,G^i) = p^{k-1} \cdot \ell\) for every \(i\) such that
    \(k-1 = \min\{j \in \natnum: i/s \le p^j\}.\)
    In particular,
    \begin{equation}
        \label{eq:order-pow-2}
        \order(F,G^{p^{k-1} \cdot s})=p^{k-1} \cdot \ell.
    \end{equation}
    Let \(t\in\natnum\) such that
    \(k=\min\{j\in\natnum:t/s \le p^j\},\)
    or equivalently, \(p^{k-1} \cdot s < t \le p^k \cdot s\).
    An application of \cref{eq:order-pow-1} yields that
    \begin{equation}
        \label{eq:order-pow-3}
        F^{p^k \cdot \ell} = G^{p^k \cdot s} \cdot H^{p^k} + 1,
    \end{equation}
    which implies
    \[F^{p^k \cdot \ell} \equiv 1 \pmod{G^{p^k \cdot s}},\]
    resulting in \(\order(F,G^{p^k \cdot s}) \mid p^k \cdot \ell\)
    Considering \cref{eq:order-pow-2}, we obtain the chain of divisibility:
    \[p^{k-1} \cdot \ell = \order(F,G^{p^{k-1}\cdot s}) \mid \order(F,G^t) \mid \order(F,G^{p^k \cdot s}) \mid p^k \cdot \ell.\]
    Hence, \(\order(F,G^t) = p^{k-1} \cdot \ell\) or \(\order(F,G^t) = p^k \cdot \ell\).
    However, based on \cref{eq:order-pow-3} we have that
    \[F^{p^{k-1} \cdot \ell} = G^{p^{k-1} \cdot s} \cdot H^{p^{k-1}} + 1\]
    with \(G \nmid H\), resulting in
    \[F^{p^{k-1} \cdot \ell} \not\equiv 1 \pmod{G^t},\]
    because \(t>p^{k-1} \cdot s\).
    Therefore, it follows that \(\order(F,G^t) = p^k \cdot \ell\).
\end{proof}

By Corollary \ref{cor:deco-bij}, the bijective system \(\Gamma_e(f,m_1) \) decomposes into smaller bijective systems \(\Gamma_e(f,g_i) \) for some polynomials \(g_i \mid m_1\). We now show that the largest cycle in the dynamics of \(\Gamma_e(f,g_i)\) is precisely the order \(\order_e(f,g_i)\).


\begin{lemma}
    \label{lem:order-period-1}
    Let \(g \in \pring{R_e}\) be a divisor of \(m_1\), and let \(z \in \pring{R_e}\) be a polynomial coprime with \(g\).
    Then \(\per(\Gamma_e(f,g), z) = \order_e(f,g)\). In particular, \(\per(\Gamma_e(f,g), 1) = \order_e(f,g)\).
\end{lemma}
\begin{proof}
    Recall that \(\Gamma_e(f,g)\) is bijective and let \(\ell = \per(\Gamma_e(f,g), 1)\).
    From \cref{the:orbit-1} we obtain that \(\ell\) is the largest cycle length in the dynamics of \(\Gamma_e(f,g)\), and we have \(f^{\ell} \equiv 1 \pmod{(g, p^e)}\).
    By the principle of minimality of \(\ell\), we conclude that \(\order_e(f,g) = \ell\). Since \(z\) is coprime with \(g\) in \(\pring{R_e}\), we have that \(f^{i} \equiv 1 \pmod{(g, p^e)}\) if and only if \(f^{i} z \equiv z \pmod{(g, p^e)}\). Thus, \(\per(\Gamma_e(f,g), z)=\ell\).
\end{proof}

Suppose that the factorization of \(m_1\) into primary and mutually coprime polynomials is given as
\[m_1(x) = g_1(x) \cdots g_r(x),\]
such that \(\bar{g}_i = G_i(x)^{k_i}\), where the polynomials \(G_i \in \pring{R_1}\) are irreducible.
By \cref{lem:cset-base,the:qureshi}, \(\cset(\Gamma_1(f,G_i))\) is of the form
\[\cset(\Gamma_1(f,\bar{g}_i)) = \{1, \ell, p\ell, \dots, p^\tau\ell\},\]
where \(\ell = \order(\bar{f},G_i)\) and \(\tau\) is specified by \cref{the:order-pow}.
Moreover, by \cref{the:omega}, we have that
\[\cset(\Gamma_e(f,g_i)) = \{1, \ell, p\ell, \dots, p^{\tau'}\ell\},\]
where \(\tau' = \tau + \kappa\) for some \(0 \le \kappa < e\).
Hence, if \(\varepsilon \in \intnum^+\) with \(1 \le \varepsilon \le e\), then we obtain \(\cset(\Gamma_\varepsilon(f,g_i)) \subseteq \cset(\Gamma_e(f,g_i))\).
The system \(\Gamma_e(f,m_1)\) is the product of the individually lifted systems \(\Gamma_e(f, g_i)\), which immediately gives that
\(\cset(\Gamma_\varepsilon(f,m_1)) \subseteq \cset(\Gamma_e(f,m_1))\).
In fact, the following result shows that the dynamics of an \gls{lfds} \(\Gamma_\varepsilon(f,m_1)\) are embedded in the dynamics of \(\Gamma_e(f,m_1)\).
\begin{lemma}
    \label{lem:z-lift}
  Let \(\varepsilon \in \intnum^+\) satisfy \(1 \le \varepsilon \le e\), and let \(k = e - \varepsilon\).
    Then the dynamics of \(\Gamma_{\varepsilon}(f,m)\) embeds into the dynamics of \(\Gamma_{e}(f,m)\), i.e. there exists an injective homomorphism between their associated functional graphs. In particular, for every \(z \in  \qring{R_{\varepsilon}[x]}{(m)} = \qring{R[x]}{(m,p^\varepsilon)} \), the period and preperiod are preserved, that is,
    \[
    \per(\Gamma_{\varepsilon}(f,m), z) = \per(\Gamma_{e}(f,m), p^k z)
    \quad \text{and} \quad
    \pper(\Gamma_{\varepsilon}(f,m),z) = \pper(\Gamma_{e}(f,m), p^k z).
    \]
\end{lemma}


\begin{proof}

Note that for $z_1,z_2 \in R[x]$, the congruence \(z_1\equiv z_2 \pmod{m,p^\varepsilon}\) implies \(p^k z_1\equiv p^k z_2 \pmod{m,p^e}\). Thus, the linear map \(\varphi: \qring{R[x]}{(m,p^\varepsilon)} \to  \qring{R[x]}{(m,p^e)} \) given by \(\varphi(z)=p^k z\) is well defined. Moreover, $\Gamma_e(f,m)(\varphi(z)) = f\cdot (p^k\cdot z) = p^k (\cdot f \cdot z) = \varphi (\Gamma_\varepsilon(f,m)(z))$. Thus, the following diagram commutes:

\[
\begin{tikzcd}
\qring{\pring{R}}{(m,p^\varepsilon)} \arrow[r, "\varphi"] \arrow[d, "{\Gamma_\varepsilon(f,m)}"']
& \qring{\pring{R}}{(m,p^e)} \arrow[d, "{\Gamma_e(f,m)}"] \\
\qring{\pring{R}}{(m,p^\varepsilon)} \arrow[r, "\varphi"]
& \qring{\pring{R}}{(m,p^e)}
\end{tikzcd}
\]

Since \(\varphi\) is well defined and commutes with the maps, it suffices to show that \(\varphi\) is injective in order to obtain the desired embedding. Let $z \in R[x]$ satisfy \(p^k z \equiv 0 \pmod{m,p^e} \). Equivalently, \(  p^k z(x) = a(x) m(x) + p^e b(x)\) for some \(a,b \in R[x]\). Then, we have

\begin{equation}
    \label{eq:congruence_modm}
    p^k ( z-p^\varepsilon b) \equiv 0 \pmod{m}.
\end{equation}

By Lemma \ref{lem:Z-free-modules}, since \(R\) is torsion-free as \(\mathbb{Z}\)-module and \( m \in \pring{R}\) is monic, \( \qring{\pring{R}}{(m(x))} \) is also torsion-free. Thus, we can cancel the factor \(p^k\) in the congruence \ref{eq:congruence_modm} to obtain \(z-p^\varepsilon b \equiv 0 \pmod{m}\), and hence \(z\equiv 0 \pmod{m,p^\varepsilon}\). Then \(\varphi\) has trivial kernel, and hence it is injective.

\end{proof}

A possible approach to compute \(\cset(\Gamma_e(f,m))\) could be as follows:
\begin{enumerate}
    \item Compute the \(f\)-decomposition of \(m\) in \(R_1\), resulting in \(m \equiv M_1 \cdot M_2 \pmod{p}\).
    \item Factor \(M_1(x) = \prod_{i=1}^r G_i(x)^{k_i}\) into irreducible and mutually coprime polynomials \(G_i \in \pring{R_1}\) with \(k_i > 0\).
    \item Lift the factorization of \(M_1\) to \(R_e\), such that \(m_1(x) = g_1(x) \cdots g_r(x)\) with \(\bar{g}_i = G_i^{k_i}\).
    \item Determine \(\fgraph(\Gamma_e(f,g_i))\), \(i=1,\dots,r\)
    \item Compute the product \(\fgraph(\Gamma_e(f,m_1)) \isomorphic \bigotimes_{i=1}^r \fgraph(\Gamma_e(f,g_i))\), which contains information about \(\cset(\Gamma_e(f,m))\).
\end{enumerate}
However, several aspects have to be considered:
\begin{enumerate}
    \item In this work, we are only interested in the cycle lengths that appear in \(\fgraph(\Gamma_e(f,m))\).
    Computing the product of the individual systems \(\fgraph(\Gamma_e(f,g_i))\) would require us to also properly consider the counts of the cycles, adding complexity
    to the overall analysis.
    \item The factorization of \(M_1\) has to be lifted to a possibly large power of \(p\), again increasing the complexity.
\end{enumerate}

To address the first problem, we introduce the \(\lcm\)-product.
\begin{definition}
    \label{def:lcm-product}
    \emph{(\(\lcm\)-Product)}
    Let \(\set{X}\) and \(\set{Y}\) be two sets.
    We define the \(\lcm\)-product of \(\set{X}\) and \(\set{Y}\) as follows
    \[\set{X} \odot \set{Y} = \{\lcm(x,y) : x \in \set{X}, y \in \set{Y}\}.\]
\end{definition}
An \(\lcm\)-product of a series of sets is denoted by \(\bigodot\).
We observe that it is easy to verify that \(\odot\) is commutative and distributive over the set union operation.
That is, if \(\set{A}, \set{B}\) and \(\set{C}\) are sets, then we have
\(\set{A} \odot (\set{B} \cup \set{C}) = (\set{A} \odot \set{B}) \cup (\set{A} \odot \set{C}).\)
Moreover, \(\set{A} \odot \{1\} = \set{A}\).
\begin{lemma}
    \label{lem:lcm-product}
    Let \(g_1(x), \dots, g_r(x)\) be as in \cref{cor:deco-bij}.
    Then \(\cset(\Gamma_e(f,m_1)) = \bigodot_{i=1}^r \cset(\Gamma_e(f,g_i))\).
\end{lemma}
\begin{proof}
    From \cref{cor:deco-bij}, we obtain a decomposition
    \[\Gamma_e(f,m_1) \isomorphic \bigotimes_{i=1}^r \Gamma_e(f,g_i).\]
    Since the maps \(\Gamma_e(f,g_i)\) are bijective, their functional graphs consist of sums of cycles.
    Hence, we can apply \cref{eq:closed-form} to compute \(\fgraph(\Gamma_e(f,m_1))\), such that \(\cset(\Gamma_e(f,m_1))\) consists of the least common multiples
    of all combinations of cycle lengths in \(\cset(\Gamma_e(f,g_i))\).
\end{proof}
\begin{lemma}
    \label{lem:irred-union}
    Let \(g\) be a divisor of \(m_1\) such that \(\bar{g} = G(x)^k\) and \(G \in \pring{R_1}\) is irreducible.
    Then for any \(e > 1\), it holds that \(\cset(\Gamma_e(f,g)) = \cset(\Gamma_{e-1}(f,g)) \cup \{\order_e(f,g)\}\).
\end{lemma}
\begin{proof}
    Let \(\ell = \order(\bar{f},G)\).
    By \cref{the:qureshi,the:order-pow}, we have that
    \[\cset(\Gamma_1(f,g)) = \{1, \ell, p \ell, \dots, p^{\tau_1} \ell\}\]
    for some \(\tau_1 \in \natnum\).
    Moreover, by \cref{the:omega}, we have that
    \[\cset(\Gamma_{e-1}(f,g)) = \{1, \ell, p \ell, \dots, p^{\tau'} \ell\},\]
    where \(\tau' = \tau + t\) for some \(1 \le t < e-1\).
    Applying \cref{the:omega} again, we obtain
    \begin{equation}
        \label{eq:irred-union-1}
        \cset(\Gamma_e(f,g)) = \{1, \ell, p\ell,\dots,p^\tau\ell\},
    \end{equation}
    where \(\tau = \tau'\) or \(\tau = \tau'+ 1\).
    Hence, \(\cset(\Gamma_e(f,g)) = \cset(\Gamma_{e-1}(f,g)) \cup \{p^\tau \ell\}\).
    Moreover, by \cref{the:orbit-1}, we can conclude that \(p^\tau \ell = \per(\Gamma_e(f,g), 1)\).
    Finally, by \cref{lem:order-period-1}, we obtain \(\order_e(f,g) = p^\tau \ell\).
\end{proof}
We are now prepared to iteratively construct \(\cset(\Gamma_e(f,m))\).

\begin{theorem}
    \label{the:cset-lcm-construction}
    Let \(g_1, \dots, g_r \in \pring{R_e}\) as in \cref{cor:deco-bij} and \(\bar{g}_i(x) = G_i(x)^{k_i}\), \(G_i\) irreducible.
    Define \(\set{H} := \bigodot_{i=1}^r \{1, \order_e(f,g_i)\}\).
    Then for any \(e > 1\), it holds that \(\cset(\Gamma_e(f,m)) = \cset(\Gamma_{e-1}(f,m)) \cup \set{H}\).
\end{theorem}

\begin{proof} 
    Since \(\cset(\Gamma_e(f,m)) = \cset(\Gamma_e(f,m_1))\), we only have to consider the bijective subsystem defined by \(m_1\).
    
    Let \(\cset:= \cset(\Gamma_e(f, m)) \) and \(\cset':= \cset(\Gamma_{e-1}(f,m)) \cup \set{H}\). We prove that \(\cset = \cset'\) by showing both inclusions. From \cref{lem:lcm-product}, we have that
    \(\cset(\Gamma_e(f, m_1)) = \bigodot_{i=1}^r \cset(\Gamma_e(f,g_i))\).
    Moreover, by \cref{lem:irred-union}, we obtain
    \begin{equation*}
        \cset = \bigodot_{i=1}^r \cset(\Gamma_{e-1}(f, g_i)) \cup \{\order_e(f,g_i)\}.
    \end{equation*}
    Let \(\set{I} = \{1, \dots, r\}\) and \(\set{S} \subseteq \set{I}\).
    For each \(i \in \set{I}\), define
    \[
        \set{D}_i(\set{S}) =
            \begin{cases}
                \{\order_e(f,g_i)\} \text{ if } i \in \set{S},\\
                \cset(\Gamma_{e-1}(f,g_i)) \text{ else.}
            \end{cases}
    \]
    Using the distributivity law, it follows that
    \begin{equation}
    	\label{eq:odot-dist-2}
        \cset = \bigcup_{\set{S}\subseteq \set{I}} \bigodot_{i=1}^r \set{D}_i(\set{S}) = \cset(\Gamma_{e-1}(f,m_1))   \cup \left( \bigcup_{\substack{\set{S}\subseteq\set{I} \\ \set{S}\neq \emptyset}} \bigodot_{i=1}^r \set{D}_i(\set{S}) \right).
    \end{equation}
    
On the other hand, since \(1\in \cset(\Gamma_{e-1}(f, m_i))\) we have

	\begin{equation}
		\label{eq:cset_prime}
 		\cset' =  \cset(\Gamma_{e-1}(f, m_i)) \cup \bigodot_{i=1}^{r}\{1, \order_{e}(f,g_i)\} =  \cset(\Gamma_{e-1}(f, m_i)) \cup \left\{ \mathop{\mathrm{lcm}}_{k\in K} \order_{e}(f,g_k) : K \subseteq I, K\neq \emptyset  \right\} 
	\end{equation}

To prove the inclusion \(\cset' \subseteq \cset\), consider \(K\subseteq I\) with \(K\neq \emptyset\) and \(h=\mathop{\mathrm{lcm}}_{k\in K} \order_e(f,g_k)\).
Since \(1\in  \cset(\Gamma_{e-1}(f,g_i))\) for every \(i\in \set{I} \setminus \set{K}\), we have that \(h \in  \bigodot_{i=1}^r \set{D}_i(\set{K})\subseteq \cset \).
Thus, \(\cset' \subseteq \cset\).

In order to prove the inclusion \(\cset \subseteq \cset'\), it suffices to prove that for each \(S\subseteq I, S\neq \emptyset\) the following inclusion holds 
	\begin{equation}
		\label{eq:cset_inclusion}
		\bigodot_{i=1}^r \set{D}_i(\set{S}) \subseteq  \cset(\Gamma_{e-1}(f, m_1)) \cup \left\{ \mathop{\mathrm{lcm}}_{k\in K} \order_{e}(f,g_k) : K \subseteq I, K\neq \emptyset  \right\}.
	\end{equation}

By Lemma \ref{lem:irred-union} and Equation \eqref{eq:irred-union-1}, \( \cset(\Gamma_{e-1}(f,g_i))=\{1,\ell_i,p\ell_i, \cdots, p^{\tau_i'}\ell_i\} \)
and \( \order_e(f,g_i)=p^{\tau_i}\ell_i \) where \(\tau_i \in \{ \tau_i', \tau_i'+1 \}  \).
Let \( \set{Y}:=\set{I}\setminus\set{S}\).
By definition of the lcm-product, for each \(h \in  \bigodot_{i=1}^r \set{D}_i(\set{S})\) there is a subset \(\set{S}_0 \subseteq \set{S}\) such that
\(h= \lcm\{ p^{\alpha_i}\ell_i, p^{\tau_j}\ell_j : i \in \set{S}_0, j\in \set{Y}\}\) with \(0\leq \alpha_i \leq \tau_i'\).
Thus, \[ h = p^{\beta} \lcm\{ \ell_k : k \in \set{S}_0 \cup \set{Y} \}, \] where \(\beta = \max\{\alpha_i, \tau_j : i \in \set{S}_0, j\in \set{Y} \}  \). We consider two cases.
\begin{itemize}
    \item Case \( \beta= \alpha_{i_0}\) for some \(i_0 \in \set{S}_0\).
    Define \( (c_i)_{1\leq i \leq r}  \) by
    \[ c_i= \left\{
    \begin{array}{ll}
        p^{\alpha_i}\ell_i & \textrm{if } i \in \set{S}_0;\\
        \ell_i & \textrm{if } i \in \set{Y};\\
        1 & \textrm{otherwise.}
    \end{array}
    \right. \]
    Since \(\alpha_{i_0} = \max\{\alpha_i: i \in \set{S_0}\} \) we have that \(h = p^{\alpha_{i_0}} \lcm\{ \ell_k : k \in \set{S}_0 \cup \set{Y} \} = \lcm\{c_i: 1\leq i \leq r\} \in \bigodot_{i=1}^{r} \cset(\Gamma_{e-1}(f,g_i)) =  \cset(\Gamma_{e-1}(f,m_1)) \).
    
    \item Case \( \beta= \tau_{j_0}\) for some \(j_0 \in \set{Y}\) and \( \tau_{j_0} > \max\{\alpha_i: i \in \set{S}_0\}\).
    We distinguish two subcases.
    \begin{itemize}
        \item If \(\tau_{j_0}< \tau_{i_0}\) for some \(i_0\in \set{S}_0\), then \(\tau_{j_0}\leq \tau_{i_0}'\) and we consider \( (c_i)_{1\leq i \leq r}  \) defined by
        \[ c_i= \left\{
        \begin{array}{ll}
            p^{\tau_{j_0}}\ell_{i_0} & \textrm{if } i=i_0;\\
            \ell_i & \textrm{if } i \in \set{S}_0\cup \set{Y}, i\neq i_0;\\
            1 & \textrm{otherwise.}
        \end{array}
        \right. \]
        We have that \(h = p^{\tau_{j_0}} \lcm\{ \ell_k : k \in \set{S}_0 \cup \set{Y} \} = \lcm\{c_i: 1\leq i \leq r\} \in \bigodot_{i=1}^{r} \cset(\Gamma_{e-1}(f,g_i)) =  \cset(\Gamma_{e-1}(f,m_1)) \).
        \item If \(\tau_{j_0}\geq \tau_{i}\) for all \(i\in \set{S}_0\), we have that \( h= p^{\tau_{j_0}}\lcm\{ \ell_k : k \in \set{S}_0 \cup \set{Y} \} = \lcm\{\order_e(f,g_k):k \in \set{S}_0 \cup \set{Y}   \} \in \left\{ \mathop{\mathrm{lcm}}_{k\in K} \order_{e}(f,g_k) : K \subseteq I, K\neq \emptyset  \right\}  \)
    \end{itemize}
\end{itemize}
Thus, in all cases, \(h\) belongs to the right-hand side of \eqref{eq:cset_inclusion}, and the inclusion follows.
\end{proof}

To tackle the second problem mentioned above, we will show that \(\order_e(f,g_i)\) can be determined based on the factorization of \(\bar{m}\) into irreducible polynomials
in \(\pring{R_1}\) such that we do not need to explicitly lift the factorization to larger rings.
\begin{lemma}
    \label{lem:g-hat-field}
    Let \(G \in \pring{R_1}\) be a divisor of \(\bar{m}_1\) and set \(\hat{G}=\bar{m}/G\).
    Then \(\per(\Gamma_1(f,m), \hat{G}) = \order(\bar{f},G)\) and \(\pper(\Gamma_1(f,m), \hat{G}) = 0\).
\end{lemma}
\begin{proof}
    Let \(k = \order(\bar{f},G)\).
    By \cref{def:order}, we have \(\bar{f}^k \equiv 1 \pmod{G}\).
    Multiplying both sides by \(\hat{G}\), we obtain \(\bar{f}^k \cdot \hat{G} \equiv \hat{G} \pmod{m}\).
    Therefore, if \(\ell=\per(\Gamma_1(f,m), \hat{G})\), then by the principle of minimality of \(\ell\), it follows that \(\ell \mid k\).
    Moreover, it is clear that \(\pper(\Gamma_1(f,m), \hat{G}) = 0\).
    
    On the other hand, we then obtain \(\bar{f}^\ell \cdot \hat{G} \equiv \hat{G} \pmod{m}\).
    The polynomial \(\hat{G}\) is nonzero, implying that \(\bar{f}^\ell \equiv 1 \pmod{G}\).
    By the principle of the minimality of \(k\), we have \(k \mid \ell\).
    From \(\ell \mid k\) and \(k \mid \ell\), it follows that \(\ell = k\).
\end{proof}
As we see in the following proposition, a similar property holds when switching from \(R_1\) to \(R_e\). 


\begin{proposition}
    \label{pro:g-hat}
Let \(\bar{m}=M_1 M_2\) be the \(\bar{f}\)-decomposition of \(\bar{m}\) with \(\gcd(\bar{f},M_1)=1\). Let \(G \in \pring{R_1}\) be a divisor of \(M_1\), and set \(\hat{G} = \bar{m}/G\). Suppose that \(\gcd(G,\hat{G}) = 1\). Let \(g, \hat{g} \in \pring{R_e}\) be monic coprime polynomials such that \[ \pi_e(m) = g \hat{g}, \quad G \equiv g \pmod{p}, \quad \text{and} \quad \hat{G} \equiv \hat{g} \pmod{p}. \]
Let \( \widetilde{g} \in \pring{R}\) be any polynomial such that \( \widetilde{g} \equiv \hat{G} \pmod{p} \).  Then \(\per\bigl(\Gamma_e(f,m), \widetilde{g}^{p^{e-1}}\bigr) = \order_e(f,g)\) and \(\pper\bigl(\Gamma_e(f,m), \widetilde{g}^{p^{e-1}}\bigr) = 0\).
\end{proposition}

\begin{proof}
 Consider the congruence
    \begin{equation}
    	\label{eq:cvr-1}
    	f^{k} \cdot \widetilde{g}^{p^{e-1}} \equiv  \widetilde{g}^{p^{e-1}} \pmod{(m,p^e)}.
    \end{equation}
    Since \(m \equiv g\cdot\hat{g} \pmod{p^e}\), where \(g,\hat{g} \in \pring{R_e}\) are coprime polynomials, by the Chinese remainder theorem, congruence \eqref{eq:cvr-1} is equivalent to the system of congruences
    \begin{align}
        \label{eq:cvr-2}
\left\{ \begin{array}{l}
f^{k} \cdot \widetilde{g}^{p^{e-1}} \equiv \widetilde{g}^{p^{e-1}} \pmod{(g, p^e)}; \\
f^{k} \cdot \widetilde{g}^{p^{e-1}} \equiv \widetilde{g}^{p^{e-1}} \pmod{(\hat{g}, p^e)}.
\end{array}  \right.
    \end{align}
	We have that \( \widetilde{g} \equiv \hat{g} \pmod{p} \). Raising both sides to the \(p^k\)-th power and applying \cref{lem:power-p} results in \( \widetilde{g}^{p^{e-1}} \equiv \hat{g}^{p^{e-1}} \pmod{p^e} \). Then,  \( \widetilde{g}^{p^{e-1}} \equiv 0 \pmod{(\hat{g}, p^e)}\) and the last conguence of the system \eqref{eq:cvr-2} holds.
	On the other hand, since \(G \equiv g \pmod{p}\), \(\hat{G} \equiv \widetilde{g} \pmod{p}\) and \(\gcd(G,\hat{G})=1\), we have that \(g\) and \(\pi_e(\widetilde{g})\) are coprime modulo \(p\). By \cref{lem:lifting-Bezout}, they are also coprime in \(\pring{R_e}\). Thus, \(\widetilde{g}\) is invertible in \(\qring{\pring{R_e}}{\ideal{g}}\) and the first congruence in the system \eqref{eq:cvr-2} is equivalent to
	\begin{equation}
		\label{eq:cvr-3}
		f^{k} \equiv 1 \pmod{(g, p^e)}
	\end{equation}
We conclude that a positive integer \(k\) satisfies the congruence \eqref{eq:cvr-1} if and only if it satisfies the congruence \eqref{eq:cvr-3}. Since \(\overline{g}=G\) divides \(M_1\) and \(M_1\) is coprime with \(\bar{f}\), then \(f\) and \(g\) are coprime modulo \(p\). By \cref{lem:lifting-Bezout} they are also coprime modulo \(p^e\) and the congruence \eqref{eq:cvr-3} has solution. In particular, by the equivalence of the congruences \eqref{eq:cvr-1} and \eqref{eq:cvr-3}, their least solutions coincide; that is, \(\per\bigl(\Gamma_e(f,m), \widetilde{g}^{p^{e-1}}\bigr) = \order_e(f,g)\). Note that the existence of solution of \eqref{eq:cvr-1} also implies \(\pper\bigl(\Gamma_e(f,m), \widetilde{g}^{p^{e-1}}\bigr) = 0\).
\end{proof}

\begin{example}
    \label{ex:1}
    In this example, let \(f(x) = x^{13}+x\), \(m(x) = x^{14}-1\), and \(p=2\).
    We have that the \gls{lfds} \(\sys(\qring{\pring{R_e}}{\ideal{m}}, \Gamma_e(f,m))\) is isomorphic to \(\sys(R_e^{n}, \Gamma_e(\mat{A}))\), where \(\pi_e(\mat{A})\) is a circulant matrix
    with first row \((0, 1, 0, \dots, 0, 1)\).
    Furthermore, let \(t = 1\) and \(u=7\) such that \(n=p^tu\) with \(p \nmid u\).
    Setting \(M_2(x) = \gcd(x^u-1, \bar{f}(x))^{p^t} = (x+1)^2\) and \(M_1(x) = \bar{m}(x) / M_2(x) = G_1(x)^2 G_2(x)^2\) with
    \(G_1(x) = x^3+x+1\) and \(G_2(x) = x^3+x^2+1\), it is easy to verify that \(\bar{m} = M_1M_2\) is the \(f\)-decomposition of \(m\) in \(\pring{R_1}\).
    Hence, if \(\bmat{A}=\pi_1(\mat{A})\) and \(\mathcal{M}_e(\mat{A})\) is \(R_e^n\) viewed as a \(\pring{R_e}\)-module where \(f\) acts as multiplication by \(\pi_e(\mat{A})\),
    then we have a decomposition of \(\mathcal{M}_1(\mat{A})\) as follows
    \[\mathcal{M}_1(\mat{A}) \isomorphic \qring{\pring{R_1}}{\ideal{(x+1)^2}} \times \qring{\pring{R_1}}{\ideal{(x^3+x+1)^2(x^3+x^2+1)^2}},\]
    leading to a decomposition of the system as
    \begin{align*}
       \sys(\qring{\pring{R_1}}{\ideal{x^{14}-1}}, \Gamma_1(f,m)) &\isomorphic \underbrace{\sys(\qring{\pring{R_1}}{\ideal{(x+1)^2}}, \Gamma_1(f,M_2))}_{\text{nilpotent}} \\
       &\otimes \underbrace{\sys(\qring{\pring{R_1}}{\ideal{(x^3+x+1)^2(x^3+x^2+1)^2}}, \Gamma_1(f,M_1))}_{\text{bijective}}.
    \end{align*}
    It is immediate that
    \begin{enumerate}
        \item \(\order(\bar{f},G_1) = 7\), \(\order(\bar{f},G_1^2) = 14\);
        \item \(\order(\bar{f},G_2) = 7\), \(\order(\bar{f},G_2^2) = 14\).
    \end{enumerate}
    Consequently, \(\cset(\Gamma_1(f,m)) = \{1, 7, 14\}\).
    Let \(\hat{G}_1(x) = (x^{14}-1)/G_1(x) = (x+1)^2 (x^3 + x + 1) (x^3 + x^2 + 1)^2\).
    \Cref{tab:ex1-iterations} lists the orbit of \(\hat{G}_1(x)\), namely the sequence of iterations \(f^k \hat{G}_1 \pmod{(m,p)}\) for increasing \(k\).
    One can see that \(f \hat{G}_1 \equiv f^8 \hat{G}_1 \pmod{(m,p)}\), which shows that \(\hat{G}_1\) lies on a cycle of length 7, verifying \cref{lem:g-hat-field}.
    Similarly, if we set \(G_i' = G_i^2\), \(i=1,2\), one can easily verify that the cycle in the orbit of \(\hat{G}_1'(x) = (x^{14}-1)/G_1(x)^2 = (x+1)^2(x^3+x^2+1)^2\) has length 14.
    \begin{table}[hbt]
        \centering
        \begin{tabular}{ll}
            \hline
            \(k\) & State \\
            \hline
             &\\[-1em]
             1 &  \(x^{13}+x^{12}+x^9+x^7+x^6+x^5+x^2+1\) \\
             2 & \(x^{12}+x^{11}+x^{10}+x^7+x^5+x^4+x^3+1\) \\
             3 & \(x^{12}+x^{10}+x^9+x^8+x^5+x^3+x^2+x\) \\
             4 & \(x^{13}+x^{10}+x^8+x^7+x^6+x^3+x+1\) \\
             5 & \(x^{13}+x^{12}+x^{11}+x^8+x^6+x^5+x^4+x\) \\
             6 & \(x^{13}+x^{11}+x^{10}+x^9+x^6+x^4+x^3+x^2\) \\
             7 & \(x^{11}+x^9+x^8+x^7+x^4+x^2+x+1\) \\
             8 & \(x^{13}+x^{12}+x^9+x^7+x^6+x^5+x^2+1\) \\
             \hline
        \end{tabular}
        \caption{Iterations of \(\hat{G}_1(x)\).}
        \label{tab:ex1-iterations}
    \end{table}

        Now, let us fix \(e=2\) and consider \(g_1 \in \pring{R_2}\) such that \(\bar{g}_1 = G_1'\).
        To compute \(\order_2(f,g_1)\), we use \cref{pro:g-hat} and choose \(\widetilde{g}_1\in\pring{R}\) with coefficient in \([0,p)\) such that \(\widetilde{g}_1 \equiv \hat{G}_1' \pmod{p}\).
        Then, we analyze the orbit of \(\widetilde{g}_1^p\), i.e. \(f^k\widetilde{g}_1^2\pmod{(m,p^2)}\) for increasing \(k\).
        The orbit is presented in \cref{tab:ex1-iter-1}.
        We observe that \(f\widetilde{g}_1^2\equiv f^{15}\widetilde{g}_1^2 \pmod{(m,p^2)}\).
        Therefore, \(\per(\Gamma_2(f,m), \widetilde{g}_1^2) = 14\), from which we can conclude that \(\order_2(f,g_1) = 14\),
        and further \(\omega(\Gamma_2(f,g_1),7) = \omega(\Gamma_1(f,g_1), 7)\); see \cref{def:omega}.
        A similar behavior can be observed in case of \(\hat{G}_2'\),
        such that \(\cset(\Gamma_2(f,m)) = \cset(\Gamma_1(f,m)) = \{1, 7, 14\}\).

        Next, let us fix \(e=3\), consider \(g_1\in\pring{R_3}\) such that \(\bar{g}_1 = G_1'\), and examine the orbit of \(\widetilde{g}_1^{p^2}\) to obtain \(\order_3(f,g_1)\).
        The orbit is given in \cref{tab:ex1-iter-2}.
        As is shown, we have that
        \[f \cdot \widetilde{g}_1^{4} \equiv f^{29} \cdot \widetilde{g}_1^4 \not\equiv f^{15} \cdot \widetilde{g}_1^{4} \pmod{(m,p^3)}.\]
        Thus, \(\order_3(f,g_1) = 28\), and \(\omega(\Gamma_3(f,g_1), 7) = \omega(\Gamma_2(f,g_1), 7) + 1\).
        Similar behavior can be observed for \(\hat{G}_2'\).
        As a result, \(\cset(\Gamma_3(f,m)) = \cset(\Gamma_2(f,m)) \cup \{28\} = \{1,7,14,28\}\).
    \begin{table}[hbt]
        \centering
        \begin{tabular}{ll}
                \hline 
                 \(k\) & State \\
                 \hline
                 &\\[-1em]
                 1 & \(3 x^{13}+x^9+x^7+x^5+2 x^3\) \\
                 2 & \(3 x^{12}+x^{10}+2 x^8+2 x^6+3 x^4+2 x^2+3\) \\
                 \(\vdots\) & \(\vdots\) \\
                 14 & \(2 x^{12}+2 x^{10}+3 x^8+2 x^6+3 x^4+3 x^2+1\) \\
                 15 & \(3 x^{13}+x^9+x^7+x^5+2 x^3\) \\
                 \hline
            \end{tabular}
            \caption{Orbit of \(\widetilde{g}_1^2\) (\(e=2\)).}
            \label{tab:ex1-iter-1}
    \end{table}
    \begin{table}[hbt]
        \centering
        \begin{tabular}{ll}
                \hline
                 \(k\) & State \\
                 \hline
                 &\\[-1em]
                 1 & \(3 x^{13}+4 x^{11}+x^9+x^7+x^5+6 x^3\) \\
                 2 & \(7 x^{12}+5 x^{10}+2 x^8+2 x^6+7 x^4+6 x^2+3\) \\
                 \(\vdots\) & \(\vdots\) \\
                 15 & \(7 x^{13}+4 x^{11}+x^9+5 x^7+x^5+2 x^3+4 x\) \\
                 \(\vdots\) & \(\vdots\) \\
                 28 & \(6 x^{12}+6 x^{10}+3 x^8+6 x^6+3 x^4+3 x^2+5\) \\
                 29 & \(3 x^{13}+4 x^{11}+x^9+x^7+x^5+6 x^3\) \\
                 \hline
            \end{tabular}
            \caption{Orbit of \(\widetilde{g}_1^{4}\) (\(e=3\)).}
            \label{tab:ex1-iter-2}
    \end{table}
\end{example}

\subsection{Algorithm to Compute \texorpdfstring{\(\cset(\Gamma_e(f,m))\)}{the Set of Cycle Lengths}}
\label{ssec:alg-cycles}

\input{algorithms/cycle_lengths}
\input{algorithms/order}

We now have all the ingredients to present an algorithm for computing \(\cset(\Gamma_e(f,m))\).
The procedure is given in \cref{alg:cycle-lengths}.
In what follows, \(O^{\sim}\) denotes soft-\(O\) complexity, suppressing polylogarithmic factors.
Furthermore, let \(\mathcal{M}_e(n)\) be the algorithmic complexity of multiplying two polynomials of degree less than \(n\) over \(R_e\).
Note that \(\lvert R_e \rvert = q^e\), and hence we need \(O(e\cdot\log(q))\) bits to represent an element in \(R_e\).
Recently, an efficient algorithm for polynomial multiplication has been proposed, which has a bit complexity of \(\mathcal{M}_e(n) = O(ne\cdot \log(q)\cdot\log(ne\log(q)))\),
or, in other words, \(O(b\cdot\log(b))\) bit operations in terms of the bit size \(b = O(ne \cdot \log(q))\) of the polynomial inputs \cite{Harvey2022}.
\begin{theorem}
    \label{the:alg-1}
    \Cref{alg:cycle-lengths} correctly computes the set of cycle lengths \(\cset(\Gamma_e(f,m))\) in expected time
    \(O^{\sim}\bigl(n^{1+\log(\log(n))} + n^2e^2 + (n^2+e^2 + ne)\cdot\mathcal{M}_e(n)\bigr)\).
\end{theorem}
\begin{proof}
Initially, we compute the \(f\)-decomposition of \(m\) in \(\pring{R_1}\) based on the method described in \cref{sec:hensel} to obtain \(M_1, M_2 \in \pring{R_1}\) satisfying
\(m(x) \equiv M_1(x) M_2(x) \pmod{p}\) in line 1.
Then in line 2, we factor \(M_1(x) = \prod_{i=1}^rG_i(x)^{k_i}\), where \(G_i\) irreducible and mutually coprime.
Next, we apply \cref{the:qureshi} to compute \(\cset(\Gamma_1(f,m))\) in line 3.

A fundamental algorithm for determining \(\order(x,G_i)\) is described in \cite[p.~87]{Lidl1996}, which can be generalized via substitution to
\(\order(F,G_i)\) for irreducible \(F \in \pring{\field_q}\).
We propose \cref{alg:order}, reflecting and generalizing the description presented in \cite{Lidl1996}.
Let \(\Phi_q(G) = \left\vert (\qring{\pring{\field_q}}{\ideal{G}})^\times\right\vert\).
\Cref{alg:order} relies on the fact that \(\order(\bar{f},G_i) \mid \Phi_q(G_i) = q^{\deg(G_i)} - 1\) \cite{Panario2018}.
The loops in lines 7--13 of \cref{alg:order} compute \(\order(\bar{f},G_i)\) by identifying the factors of \(\Phi_q(G_i)\) that make up the order.

We continue with the discussion of \cref{alg:cycle-lengths}.
Suppose \(m_1 \in \pring{R_e}\) such that \(\bar{m}_1 = M_1\), and \(m_1(x) = g_1(x)g_2(x) \cdots g_r(x)\) with \(\bar{g}_i = G_i^{k_i}\).
Line 4 computes the reciprocals \(\hat{G}_i\), followed by choosing polynomials \(\widetilde{g}_i \in \pring{R}\) in line 5, which are later
used to determine the value of \(\omega_{i,\varepsilon}\).
In lines 7--12, \(\cset(\Gamma_\varepsilon(f,m))\) is computed iteratively for \(\varepsilon=2,\dots,e\).
At each step, we determine whether the maximum cycle length in the bijective subsystem \(\sys(\qring{\pring{R_\varepsilon}}{\ideal{g_i}}, \Gamma_\varepsilon(f, g_i))\)
is multiplied by \(p\) or not in accordance with \cref{the:omega}.
To this end, lines 8--9 utilize \cref{pro:g-hat} to check if the length of the cycle containing the vertex \(\widetilde{g}_i(x)^{p^{\varepsilon-1}} \pmod{(m, p^{\varepsilon})}\)
in \(\fgraph(\Gamma_\varepsilon(f,m))\) increases compared to \(\widetilde{g}_i(x)^{p^{\varepsilon-2}} \pmod{(m, p^{\varepsilon-1})}\) in \(\fgraph(\Gamma_{\varepsilon-1}(f,m))\).
Let \(\ell_{\varepsilon-1} = \order_{\varepsilon-1}(f,g_i)\).
Our method, informed by \cref{pro:g-hat}, verifies the equivalence
\[f(x)^{\ell_{\varepsilon-1}} \cdot \widetilde{g}_i(x)^{p^{\varepsilon-1}} \equiv \widetilde{g}_i(x)^{p^{\varepsilon-1}} \pmod{(m, p^\varepsilon)}.\]
If true, the maximal cycle length in \(\cset(\Gamma_\varepsilon(f,g_i))\) does not increase compared to \(\cset(\Gamma_{\varepsilon-1}(f,g_i))\).
Otherwise, it increases by a factor of \(p\).
We remark that in line 8, we do not compute \(\omega(\Gamma_\varepsilon(f,g_i), \ell)\) where \(\ell = \max(\csetp(\Gamma_\varepsilon(f,g_i)))\) (see \cref{lem:irred-union}).
Instead, if \(\cset(\Gamma_1(f,g_i)) = \{1, \ell, p\cdot \ell, \dots, p^{\tau_1}\cdot\ell\}\) for some \(\tau_1 \in \natnum\),
then we compute the largest nonnegative integer \(\omega_{i,\varepsilon}\) such that \(p^{\omega_{i,\varepsilon} + \tau_1}\cdot\ell \in \cset(\Gamma_\varepsilon(f,g_i))\).
Both approaches are equivalent, as \((\omega(\Gamma_\varepsilon(f,g_i),\ell) - \omega(\Gamma_1(f,g_i), \ell)) - (\omega_{i,\varepsilon} - \omega_{i,1}) = \tau_1\).
Consequently, we set \(\omega_{i,1} = 0\) in line 6.
Once \(\omega_{i,\varepsilon}\) is updated in lines 8--9, \(\cset(\Gamma_\varepsilon(f,m))\) is computed in lines 10--12 as per \cref{the:cset-lcm-construction}.

We proceed with an analysis of the runtime complexity of \cref{alg:cycle-lengths}.
Its actual time complexity depends on the polynomial representation and the underlying algorithms for individual operations.
Certain steps may be split into multiple computations, and intermediate results can be stored for efficiency; however, such optimizations are not considered here.
In this analysis, we make typical assumptions and focus on average-case performance.
We assume the following runtime complexities for standard polynomial operations:
\begin{itemize}
    \item multiplication and division with remainder: \(O(\mathcal{M}_e(n))\) \cite{Mullen2013};
    \item factorization over \(\field_q\) (probabilistic): \(O(n\cdot \mathcal{M}_1(n)\cdot\log(nq))\) \cite{Gathen1999};
    \item modular exponentiation \(f(x)^k \pmod{m}\): \(O(\mathcal{M}_e(n)\cdot\log(k))\) via square-and-multiply \cite{Gathen2001}.
\end{itemize}
As all computations are modulo \(m\), and polynomial degrees are less than \(\deg(m)=n\). 

Lines 1--6 operate over \(R_1 \isomorphic \field_q\).
Line 1 involves factoring \(f\) and \(m\) and identifying common factors.
It has been shown that the expected number of distinct irreducible factors of a polynomial of degree \(n\) is \(O(\log(n))\) \cite{Panario2004},
yielding \(r\in O(\log(n))\).
Since identification of common factors requires iteration over the factors of \(f\) for each factor of \(m\), it has complexity of \(O(\log(n)^2)\).
Line 2 is again a factorization of \(M_1\).
Thus, lines 1--2 have complexity of \(O(n\cdot \mathcal{M}_1(n) \cdot \log(nq))\).

The complexity of line 3 is more subtle, involving three sequential steps by following \cref{the:order-pow,pro:prod-cycles}:
\begin{enumerate}
    \item Computing \(\order(\bar{f},G_i)\), where \(G_i\) is irreducible: 
    This is done via \cref{alg:order}, and its complexity is as follows.
    Lines 1--5 are basic.
    Line 6 requires the prime factorization of \(q^m-1\), with \(m = \deg(G)\).
    Assuming a precomputed lookup table, this is \(O(1)\).
    The outer loop in line 7 is analyzed using the Hardy-Ramanujan theorem: \(v \approx \log(n\cdot\log(q))\), yielding \(O(\log(n\log(q)))\) as an upper bound for the number of iterations.
    The worst case for the number of iterations of the inner loop in line 8 is given by the hypothetical case \(s=1\) such that \(q^m-1 = p_1^{r_1}\) and hence \(r_1 \propto m \cdot \log(q)\).
    Since \(m < n\), the inner loop iterates no more than \(O(n \cdot \log(q))\).
    In line 9, modular polynomial exponentiation with exponent less than \(q^m - 1 < q^n\) has complexity of no more than \(O(n \cdot \mathcal{M}_1(n) \cdot \log(q))\).
    Overall, \cref{alg:order} operates in expected time \(O(n^2 \cdot \mathcal{M}_1(n) \cdot \log(n\log(q)) \cdot \log(q)^2)\), or \(O^{\sim}(n^2 \cdot \mathcal{M}_1(n))\).
    
    \item Determining \(\order(\bar{f},G_i^{\kappa_i})\) for \(\kappa_i=1,\dots,k_i\), \(i=1,\dots,r\): 
    Considering \cref{the:order-pow}, an important observation is that for each \(G_i\), the terms \(G_i^{\kappa_i}\) can be grouped into mutually disjoint sets \(H_{i,1}, \dots, H_{i,\vartheta_i}\), \(\vartheta_i \in O(\log(n))\), 
    of exponentially increasing size, such that for any set \(H_{i,u}\), \(u=1, \dots, \vartheta_i\), there exists \(t_{i,u}\) satisfying \(\order(\bar{f}, G') = p^{t_{i,u}}\order(\bar{f}, G_i)\) for any \(G' \in H_{i,u}\).
    For example, let us consider the case \(p=2\), \(r=1\), \(f(x) = x\), and \(G_1(x)^{k_1} = (x^3+x+1)^{16}\).
    In this case, we observe that \(\vartheta_1 = 5\), and the sets \(H_{1,1}\) to \(H_{1,5}\) are listed in \cref{tab:factor-order-groups}.
    \begin{table}
        \centering
        \begin{tabular}{ll>{$}l<{$}@{${}={}$}l>{$}l<{$}@{${}\eqcolon{}$}>{$}l<{$}}
            \hline
            Sets \(H_{1,u}\) & \multicolumn{1}{l}{Factors} & \multicolumn{2}{l}{Order} & \multicolumn{2}{l}{\(t_{1,u}\)} \\
            \hline
            \multicolumn{1}{c}{}&\multicolumn{1}{c}{}&\multicolumn{2}{c}{}&\multicolumn{2}{c}{}\\[-1em]
             \(H_{1,1}\) & \(\{G_1\}\)                    & \order(\bar{f}, G_1)    & 7   & t_{1,1} & \tau_{1,1} = 0 \\
             \(H_{1,2}\) & \(\{G_1^2\}\)                  & p\order(\bar{f}, G_1)   & 14  & t_{1,2} & \tau_{1,2} = 1 \\
             \(H_{1,3}\) & \(\{G_1^3, G_1^4\}\)           & p^2\order(\bar{f}, G_1) & 28  & t_{1,3} & \tau_{1,3} = \tau_{1,4} = 2 \\
             \(H_{1,4}\) & \(\{G_1^5, \dots, G_1^8\}\)    & p^3\order(\bar{f}, G_1) & 56  & t_{1,4} & \tau_{1,5} = \cdots = \tau_{1,8} = 3 \\
             \(H_{1,5}\) & \(\{G_1^9, \dots, G_1^{16}\}\) & p^4\order(\bar{f}, G_1) & 112 & t_{1,5} & \tau_{1,9} = \cdots = \tau_{1,16} = 4 \\
             \hline
        \end{tabular}
        \caption{Sets of powers of \(G_1(x) = x^3+x+1\) with identical order for \(p=2\), \(r=1\), \(f(x)=x\), and \(k_1=16\).}
        \label{tab:factor-order-groups}
    \end{table}
    
    Our approach is to first compute \(F' \equiv f^{\order(\bar{f},G_i)} \pmod{(m, p)}\).
    This step consists of modular exponentiation and has a complexity of no more than \(O(n\cdot \mathcal{M}_1(n)\cdot\log(q))\), because \(\order(\bar{f}, G_i) \mid q^{\deg(G_i)} - 1 < q^n\).
    After that, we repeatedly perform a binary search to identify the sets \(H_{i,u}\) and the corresponding values of \(t_{i,u}\).
    The binary search has to be performed \(O(\log(n))\) times (for each set), where each search is \(O(\log(n))\), because \(k_i < n\).
    In every step of each search, we check whether \(F'^{p^{t_{i,u}}} \equiv f \pmod{(m, p)}\) is true for the current value of \(t_{i,u}\), which essentially is polynomial exponentiation and
    has complexity \(O(\mathcal{M}_1(n) \cdot \log(n) \cdot \log(p))\), because \(t_{i,u} \in O(\log(n))\).
    The entire procedure must be repeated \(r\) times for each \(G_i\).
    Recall from above that the expected number of \(r\) is \(O(\log(n))\).
    Hence, a loose upper bound of the overall expected runtime of this step is \(O(\log(n)\cdot(n\cdot \mathcal{M}_1(n) \cdot \log(q) + \log(n)^2\cdot(\mathcal{M}_1(n) \cdot \log(n) \cdot \log(p))))\),
    which simplifies to \(O(n \cdot \mathcal{M}_1(n) \cdot \log(n) \cdot \log(q))\), or \(O^{\sim}(n\cdot\mathcal{M}_1(n))\).
    
    \item Calculating orders for all remaining divisors of \(M_1(x)\):
    We must compute the set
    \[\{\lcm(o_1, \ldots, o_r) : o_i \in \{\order(\bar{f},G_i^{\kappa_i}) : \kappa_i = 1,\dots, k_i\}, 1 \le i \le r\}.\]
    The Euclidean algorithm computes \(\lcm(o_a, o_b)\) in \(O(\log(\min(o_a, \allowbreak o_b)))\) for any \(1 \le a, b \le r\).
    Clearly, \(o_i < q^n\), and thus the complexity of computing \(\lcm(o_a, o_b)\) is at most \(O(n\cdot \log(q))\).
    For each combination, \(r-1\) pairwise \(\lcm\) are needed.
    Since \(r \in O(\log(n))\), \(r\)-ary \(\lcm\) has expected time complexity of \(O(n \cdot \log(n) \cdot \log(q))\).
    This must be done for every choice of \(o_i\).
    From the discussion of the previous step, we know that we have \(\vartheta_i \in O(\log(n))\) different choices for each \(o_i\).
    Therefore, this step has an overall expected time complexity of no more than \(O(\log(n)^{\log(n)} \cdot n \cdot \log(n) \cdot \log(q))\).
    Using the identity \(\log(n)^{\log(n)} = n^{\log(\log(n))}\) and soft-\(O\) notation, we obtain \(O^{\sim}\bigl(n^{1+\log(\log(n))}\bigr)\).
\end{enumerate}
Since the first part dominates the second part, the overall expected runtime complexity of line 3 is given by
\(O^{\sim}\bigl(n^{1+\log(\log(n))}+n^2\cdot\mathcal{M}_1(n)\bigr)\).
Thus, line 3 dominates line 4 (polynomial division) and the trivial operations in lines 5--6.

The body of the loop in \cref{alg:cycle-lengths} executes \(O(e)\) iterations.
It is a well-known fact that \(\Phi_q(\bar{f}) \le q^{\deg(\bar{f})}-1\).
Therefore, line 8, dominated by polynomial exponentiation, is no worse than \(O(\mathcal{M}_e(n) \cdot \log(p^eq^n))\)
and thus grows no more than \(O((n+e)\cdot \mathcal{M}_e(n) \cdot \log(q))\).
Line 9 is not more complex than line 8; lines 10 and 12 involve straightforward set operations.
In line 11, each \(a_i\) is bounded by \(p^e\cdot(q^n-1) < q^{en}\) and has two possible options.
Since the expected number of \(r\) is \(O(\log(n))\), we have \(O(2^{\log(n)}) = O(n)\) computations of \(r\)-ary \(\lcm\),
each of which requiring \(O(\log(n))\) computations of 2-ary \(\lcm\), which in turn have complexity \(O(\log(q^{en}))\)
based on the upper bound of \(a_i\).
Therefore, line 11's time complexity grows with \(O(n^2e \cdot \log(n) \cdot \log(q))\).
Consequently, the loop body is dominated by lines 8 and 11, and the loop's complexity as a whole grows as
\(O^{\sim}(n^2e^2 + (e^2 + ne) \cdot \mathcal{M}_e(n))\).

Summing up, \cref{alg:cycle-lengths} is dominated by the order computations in line 3, as well as the loop in lines 7--13 that is dominated by lines 8 and 11,
totally resulting in an expected quasi-polynomial time complexity of
\(O^{\sim}\bigl(n^{1+\log(\log(n))} + n^2e^2 + n^2\cdot\mathcal{M}_1(n) + (e^2 + ne)\cdot\mathcal{M}_e(n)\bigr)\).
For the sake of simplicity, a slightly looser bound is given as \(O^{\sim}\bigl(n^{1+\log(\log(n))} + n^2e^2 + (n^2+e^2 + ne)\cdot\mathcal{M}_e(n)\bigr)\).
\end{proof}

There exist other algorithms that can be used to compute the set of cycle lengths.
The algorithm proposed in \cite{Wei2016} computes not only the set of cycle lengths, but also the complete set of cycles.
The stated runtime complexity is essentially of order \(n^3\) for a single cycle length.
However, their analysis neglects the need to iterate over the factors of the maximum cycle length to enumerate the full set \(\cset(\Gamma_e(f,m))\).
In our algorithm, we only compose cycle lengths that actually exist in the final set, which is central to the quasi-polynomial runtime observed.

The strength of our approach lies in avoiding costly matrix-based operations for \glspl{lfds} and leveraging polynomial factorization,
which scales more favorably than \(O(n^3)\).
A limitation is that our method does not construct every explicit cycle.
Nevertheless, in \cref{lem:z-lift,pro:g-hat}, we showed a simple method to explicitly construct a representative cycle for each cycle length that appears in the system's dynamics.

The software tool \gls{adam} \cite{Hinkelmann2011} computes all concrete cycles for a given length by solving systems of polynomial equations using Gröbner bases.
However, without prior knowledge of available cycle lengths, one must attempt every candidate length and check for solutions in the associated system.

\subsection{Reducing the Number of Factors}
As we have seen, the computation of \(\cset(\Gamma_e(f,m))\) relies on the factorization of \(m_1\) into irreducible and mutually coprime factors.
In the computation of \(\fgraph(\Gamma_1(f,m))\) in line 3 of \cref{alg:cycle-lengths}, as well as in the computations regarding \(\set{H}\) in line 10,
we have to consider all combinations of the factors.
Also, the computations of \(\omega_{i,\varepsilon}\) and \(\set{K}_i\) have to be done for each lifted factor.
There are cases in which this is unfavorable, as the following example demonstrates.

\begin{example}
    \label{exp:many-factors}
    Let \(f(x)=x^{4094}+x\), \(m(x) = x^{4095}-1\), and \(p=2\).
    The \(f\)-decomposition of \(m\) in \(\pring{R_1}\) is given by \(\bar{m} = M_1M_2\) with \(M_2(x) = x+1\) and \(M_1(x) = \bar{m}(x)/M_2\).
    Considering the factorization of \(M_1(x) = G_1(x) \cdots G_r(x)\) into irreducible and mutually coprime factors \(G_i\), we have that \(r=351\) (including the unit factor).
    However, when computing the orders \(\order(\bar{f},G_i)\), we can observe that most of the \(G_i\) have the same order.
    This is illustrated in \cref{fig:order-dist}, which shows the number of divisors \(G_i\) with the same order.
    For example, the largest fraction of divisors has order \(4095\).
    In total, there are only 22 distinct orders, which is much fewer than the 351 divisors.
    \begin{figure}[htb]
        \centering
        \includegraphics[width=\linewidth]{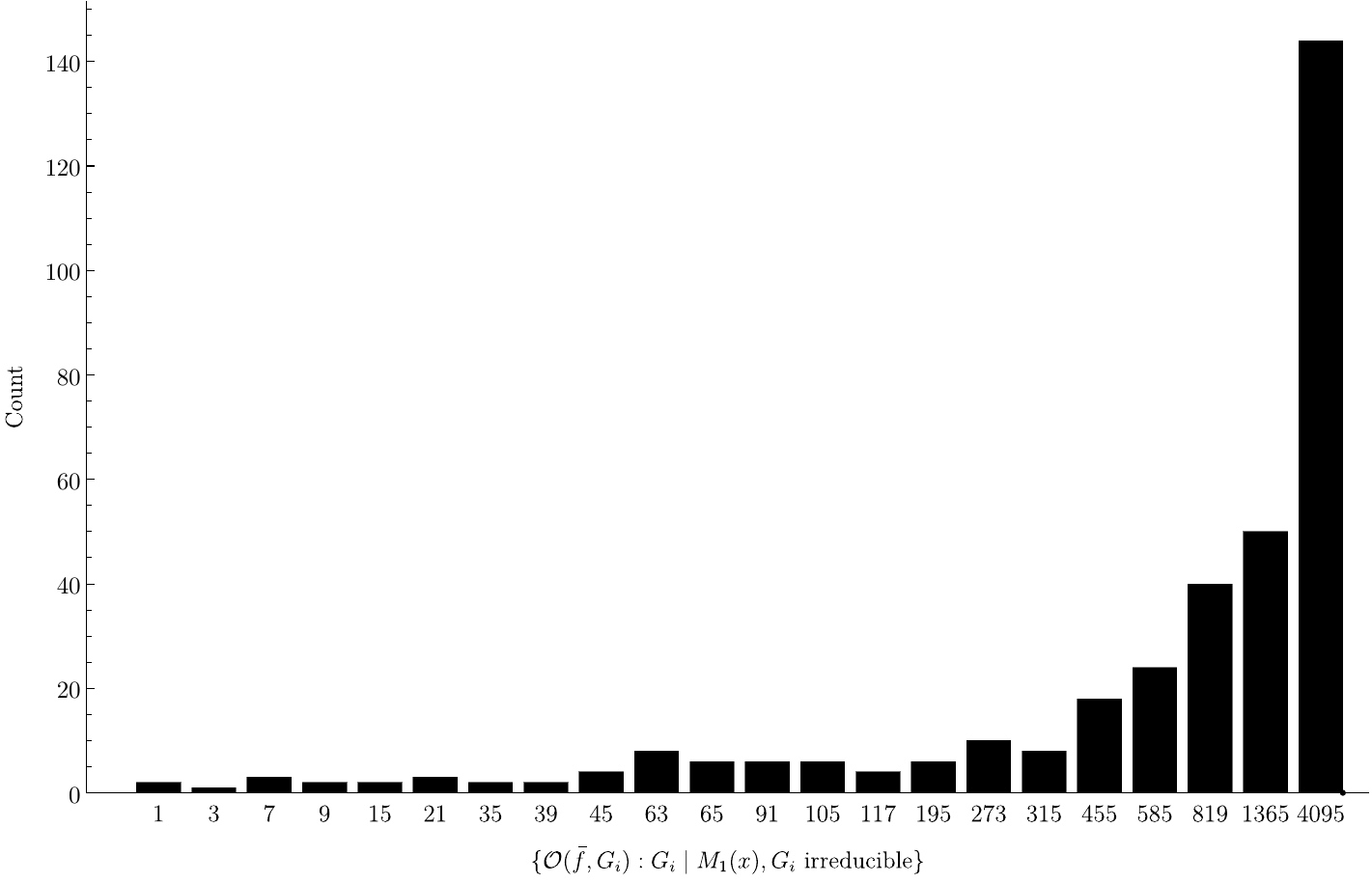}
        \caption{Number of irreducible divisors \(G_i\) of \(M_1(x)\) from \cref{exp:many-factors}, aggregated per order.}
        \label{fig:order-dist}
    \end{figure}
\end{example}

\Cref{exp:many-factors} demonstrates that for certain choices of \(f\) and \(m\), numerous divisors \(G_i\) of \(M_1\) have the same order.
In the following, we show how to exploit this property and reduce the number of factors that we need to consider when computing \(\cset(\Gamma_e(f,m))\).
For this purpose, we introduce the following definition.
\begin{definition}
    \label{def:eq-rel}
    \emph{(Order Equivalence Relation)} Let \(g_1, g_2 \in \pring{R}\) and let \(F=\bar{f}\), satisfying \(\gcd(F, \bar{g}_i) = 1\) for \(1 \le i \le 2\).
    Then the order equivalence relation \(\erel{F}\) is defined as follows:
    \[g_1 ~ \erel{F} ~ g_2 \quad \mbox{if and only if} \quad \order_1(f,g_1) = p^k\cdot \order_1(f,g_2) \text{ for some } k \in \natnum.\]
\end{definition}

Suppose that \(\bar{m}_1(x) = G_1(x) \cdots G_r(x)\), where the polynomials \(G_i(x) = G_i'(x)^{k_i}\), \(i=1,\dots,r\), such that \(G_i'\) are irreducible and mutually coprime.
Furthermore, let \(\set{B} = \{G_i' : 1 \le i \le r\}\).
Then \(\set{B}/\erel{F} = \{\set{B}_1 , \dots, \set{B}_\eta\}\) defines a partition of \(\set{B}\).
Note that every element in \(\set{B}_j\), \(1 \le j \le \eta\), has the same polynomial order over \(R_1\) with respect to \(f\).
Most notably, it is easy to check that \(\eta \le r\).
Now, we can define the polynomials
\begin{equation}
    \label{eq:theta}
    \Theta_j(x) \coloneqq \prod_{G_{i}' \in \set{B}_j} G_i'(x)^{k_i}.
\end{equation}
Clearly, the polynomials \(\Theta_j\) are mutually coprime and define subsystems \(\sys(\qring{\pring{R_1}}{\ideal{\Theta_j}}, \Gamma_1(f,\Theta_j))\),
and therefore form a decomposition of the bijective subsystem of the \gls{lfds} as follows
\begin{equation}
    \label{eq:deco-theta}
    \Gamma_1(f,m_1) \isomorphic \bigotimes_{j=1}^\eta \Gamma_1(f,\Theta_j).
\end{equation}

\begin{proposition}
    \label{pro:cset-theta-field}
    Let \(\bar{m}_1(x) = \prod_{j=1}^{\eta}\Theta_j(x)\), where the polynomials \(\Theta_j\) are defined as in \cref{eq:theta}.
    Furthermore, let \(H_j\) be a nonzero and square free divisor of \(\Theta_j\), and \(\ell_j = \order(\bar{f},H_j)\).
    Then \(\cset(\Gamma_1(f,\Theta_j)) = \{1, \ell_j, p\ell_j, \dots, p^{\tau_j}\ell_j\}\) for some \(\tau \in \natnum\).
\end{proposition}
\begin{proof}
    Since \(H_j\) is squarefree, it is of the form \(H_j = \prod_{G_i'\in\set{B}_j} G_i'\), where the polynomials \(G_i\) are irreducible.
    Assume that \(\set{B}_j = \{G_{j,1}',\dots,G_{j,r_j}'\}\).
    Moreover, by definition of \(\set{B}_j\) we have that \(\order(\bar{f},G_{j,1}') = \dots = \order(\bar{f},G_{j,r_j}') = \ell_j\).
    Suppose that \(\Theta_j(x) = \prod_{h=1}^{r_j}G_{j,h}'(x)^{k_{h}}\) and let \(G_{j,h}(x) = G_{j,h}'(x)^{k_h}\).
    By \cref{the:qureshi,the:order-pow} we have that \(\cset(\Gamma_1(f,G_{j,h})) = \{1, \ell_j, p\ell_j, \dots, p^{\tau_{j,h}}\ell_j\}\) for some \(\tau_{j,h} \in \natnum\).
    Let \(\tau_j = \max(\tau_{j,1}, \dots, \tau_{j,r_j})\).
    Then, it follows that \(\cset(\Gamma_1(f,\Theta_j)) = \{1, \ell_j, p\ell_j, \dots, p^{\tau_j}\ell_j\}\).
\end{proof}

\begin{proposition}
    \label{pro:theta-lift}
    Let \(\bar{m}_1(x) = \prod_{j=1}^\eta\Theta_j(x)\), where the polynomials \(\Theta_j\) are defined as in \cref{eq:theta}.
    Then there exists polynomials \(\theta_1, \dots, \theta_\eta \in \pring{R_e}\) with \(\bar{\theta}_j = \Theta_j\) such that
    \(\Gamma_e(f,m_1)\) decomposes into a product of bijective subsystems as follows
    \[\Gamma_e(f,m_1) \isomorphic \Gamma_e(f,\theta_1) \otimes \cdots \otimes \Gamma_e(f,\theta_\eta).\]
\end{proposition}
\begin{proof}
    The polynomials \(\Theta_j\) are mutually coprime, and thus we can apply Hensel's lemma to lift the factorization
    \(\bar{m}_1(x) = \prod_{j=1}^\eta\Theta_j(x)\) from \(R_1\) to \(R_e\), \(e\ge1\).
    Thus, there exist polynomials \(\theta_1, \dots, \theta_\eta \in \pring{R_e}\) with \(\bar{\theta}_j=\Theta_j\) such that the factorization lifts to
    \(m_1(x) = \prod_{j=1}^\eta\theta_j(x)\), and we obtain the isomorphism
    \[\qring{\pring{R_e}}{\ideal{m_1}} \isomorphic \qring{\pring{R_e}}{\ideal{\theta_1}} \times \cdots \times \qring{\pring{R_e}}{\ideal{\theta_\eta}}.\]
    Via this isomorphism, we can decompose the map \(\Gamma_e(f,m_1)\) as \(\Gamma_e(f,m_1) = \Gamma_e(f,\theta_1) \times \cdots \times \Gamma_e(f,\theta_\eta),\)
    leading to the claimed result.
\end{proof}

Let \(\tau_j\) be as in \cref{pro:cset-theta-field}.
\Cref{the:omega} yields
\(\cset(\Gamma_e(f,\theta-j)) = \{1, \ell_j, p\ell_j,\dots,p^{\tau_j+t}\ell_j\}\)
for some \(0 \le t < e\).
What we have shown so far is that the polynomials \(\theta_j\) are coprime and thus lead to a decomposition of the bijective subsystem, and also that their set of cycles has
similar structure as those of the lifted factorization of \(\bar{m}_1\) into irreducible polynomials \(G_i(x)^{k_i}\).
This allows us to use the polynomials \(\Theta_j\) instead of \(G_i^{k_i}\) in \cref{alg:cycle-lengths}, as we show in the following.
\begin{lemma}
    \label{lem:theta-union}
    Let \(m_1(x) = \prod_{j=1}^\eta\theta_j(x)\), where the polynomials \(\theta_j\) are defined as in \cref{pro:theta-lift}.
    Then for any \(e>1\), it holds that \(\cset(\Gamma_e(f,\theta_j)) = \cset(\Gamma_{e-1}(f,\theta_j))\cup\{\order_e(f,\theta_j)\}\).
\end{lemma}
\begin{proof}
    The proof of \cref{lem:theta-union} is analogous to that of \cref{lem:irred-union}, replacing \(g\) with \(\theta_j\), and we leave out the details here.
\end{proof}
\begin{theorem}
    \label{the:cset-lcm-construction-theta}
    Let \(m_1(x) = \prod_{j=1}^\eta\theta_j(x)\), where the polynomials \(\theta_j\) are defined as in \cref{pro:theta-lift}.
    Define \(\set{K}_i = \{1, \order_e(f,\theta_j)\}\) and \(\set{H}=\set{K}_1\odot\cdots\odot\set{K}_\eta\).
    Then for any \(e > 1\), it holds that \(\cset(\Gamma_e(f,m)) = \cset(\Gamma_{e-1}(f,m))\cup\set{H}\).
\end{theorem}
\begin{proof}
    Again, the proof is analogous to that of \cref{the:cset-lcm-construction}, replacing the polynomials \(g_i\) with \(\theta_j\), and we leave out the details here.
\end{proof}
As a result, we can formulate an optimized version of \cref{alg:cycle-lengths} in \cref{alg:cycle-lengths-opt}.
We still use \cref{alg:order}, and the overall complexity of \cref{alg:cycle-lengths-opt} remains the same as for \cref{alg:cycle-lengths}.
\input{algorithms/cycle_lengths_opt}

%% file: algorithms/cycle_lengths.tex
\begin{algorithm}[htb]
\begin{algorithmic}[1]
    \Require Polynomials \(f,m\in\pring{R}\), \(f\) monic, \(n=\deg(m)\), \(p\) prime, and \(e \in \intnum^+\)
    \Ensure \(\cset(\Gamma_e(f,m))\)
    \State \texttt{Compute} the \(f\)-decomposition \(m(x) = M_1(x) M_2(x)\) of \(m\) in \(\pring{R_1}\)
    \State \texttt{Factor} \(M_1(x) = G_1(x)^{k_1} \cdots G_r(x)^{k_r}\), where \(G_i\), \(i=1,\dots,r\), are irreducible and mutually coprime
    \State \texttt{Compute} \(\cset(\Gamma_1(f,m)) = \mleft\{\order(\bar{f},G) : G \mid M_1\mright\}\), marking \(\order(\bar{f},G_i^{k_i})\),\(i=1,\dots,r\) \Comment{\cref{the:qureshi}, \cref{alg:order}}
    \State \texttt{Compute} \(\hat{G}_i(x) = \frac{M_1(x)}{G_i(x)^{k_i}}\), \(i=1,\dots,r\)
    \State \texttt{Choose} \(\widetilde{g}_i \in \pring{R}\) with coefficients in \([0,p)\) such that \(\widetilde{g}_i \equiv \hat{G}_i \pmod{p}\), \(i=1,\dots,r\)
    \State \texttt{Set} \(\omega_{i,1} = 0\), \(i=1,\dots,r\)
    \For{\(\varepsilon = 2, \dots, e\)}
        \State \texttt{Compute} \(c_i \equiv f^{p^{\omega_{i,\varepsilon-1}}\cdot\order(\bar{f},G_i^{k_i})} \cdot \widetilde{g}_i^{p^{\varepsilon-1}} \pmod{(m, p^\varepsilon)}\), \(i=1,\dots,r\) \Comment{\cref{pro:g-hat}}
        \State \texttt{Compute} \(
            \omega_{i,\varepsilon} =
                \begin{cases}
                    \omega_{i,\varepsilon-1}, &\text{if } c_i \equiv \widetilde{g}_i^{p^{\varepsilon-1}} \pmod{(m, p^\varepsilon)}\\
                    \omega_{i,\varepsilon-1} + 1, &\text{else}
                \end{cases},
                ~ i=1,\dots,r
        \) \Comment{\cref{the:omega}}
        \State \texttt{Set} \(\set{K}_i = \{1,p^{\omega_{i,\varepsilon}}\cdot\order(\bar{f},G_i^{k_i})\}\), \(i=1,\dots,r\) \Comment{\cref{the:cset-lcm-construction}, using \(\order(\bar{f},G_i^{k_i})\) from line 3}
        \State \texttt{Compute} \(\set{H} = \bigodot_{i=1}^r \set{K}_i = \mleft\{\lcm(a_1, \dots, a_r) : a_i \in \set{K}_i\mright\}\)
        \State \texttt{Compute} \(\cset(\Gamma_\varepsilon(f,m)) = \cset(\Gamma_{\varepsilon-1}(f,m)) \bigcup \set{H}\)
    \EndFor
    \State \Return \(\cset(\Gamma_e(f,m))\)
\end{algorithmic}
\caption{Computes the set of cycle lengths \(\cset(\Gamma_e(f,m))\).}
\label{alg:cycle-lengths}
\end{algorithm}

%% file: algorithms/order.tex
\begin{algorithm}[!htb]
\begin{algorithmic}[1]
    \Require \(n \in \intnum^+\), \(F, G \in\pring{\field{q}}\) such that \(\gcd(F, G) = 1\) and \(G\) is irreducible with \(G(0) \ne 0\)
    \Ensure \(\order(F,G)\)
    \State \texttt{Set} \(z=1\)
    \If{\(G(x) = 1\)}
        \State\Return \(z\)
    \EndIf
    \State{Set} \(h = \deg(G)\)
    \State \texttt{Compute} \(p_1^{r_1} \cdots p_v^{r_v} = \operatorname{lookup\_prime\_factorization}(q^h - 1)\)
    \For{\(j=1,\dots,v\)}
        \For{\(r=r_j-1,r_j-2,\dots,0\)}
            \If{\(F^{p_1^{r_1}\cdots p_j^{r} \cdots p_v^{r_v}} \not\equiv 1 \pmod{G}\)}
                \State \texttt{Compute} \(z = z \cdot p_j^{r+1}\)
            \EndIf
        \EndFor
    \EndFor
    \State\Return \(z\)
\end{algorithmic}
\caption{Computes the generalized polynomial order \(\order(F,G)\) over \(\field_q\) for irreducible \(G\).}
\label{alg:order}
\end{algorithm}

%% file: algorithms/cycle_lengths_opt.tex
\begin{algorithm}[htb]
\begin{algorithmic}[1]
    \Require Polynomials \(f,m\in\pring{R}\), \(f\) monic, \(n=\deg(m)\), \(p\) prime, and \(e \in \intnum^+\)
    \Ensure \(\cset(\Gamma_e(f,m))\) 
    \State \texttt{Compute} the \(f\)-decomposition \(m(x) = M_1(x) M_2(x)\) of \(m\) in \(\pring{R_1}\)
    \State \texttt{Factor} \(M_1(x) = G_1(x)^{k_1} \cdots G_r(x)^{k_r}\), where \(G_i\), \(i=1,\dots,r\), are irreducible and mutually coprime
    \State \texttt{Set} \(\set{B} = \{G_i(x) : 1 \le i \le r\}\)
    \State \texttt{Compute} \(\set{B}/\erel{\bar{f}} = \{\set{B}_1, \dots, \set{B}_\eta\}\) \Comment{\cref{def:eq-rel}}
    \State \texttt{Compute} \(\Theta_j(x) = \prod_{G_i \in B_j} G_i(x)^{k_i}\), \(j=1,\dots,\eta\) \Comment{\cref{eq:theta}}
    \State \texttt{Compute} \(\cset(\Gamma_1(f,m)) = \mleft\{\order(\bar{f},G) : G \mid M_1\mright\}\), marking \(\order(\bar{f}, \Theta_j)\), \(j=1,\dots,\eta\) \Comment{\cref{the:qureshi}, \cref{alg:order}}
    \State \texttt{Compute} \(\hat{\Theta}_j(x) = \frac{M_1(x)}{\Theta_j(x)}\), \(j=1,\dots,\eta\)
    \State \texttt{Choose} \(\widetilde{g}_j \in \pring{R}\) with coefficients in \([0,p)\) such that \(\widetilde{g}_j\equiv\hat{\Theta}_j \pmod{p}\), \(j=1,\dots,\eta\)
    \State \texttt{Set} \(\omega_{j,1} = 0\), \(j=1,\dots,\eta\)
    \For{\(\varepsilon = 2, \dots, e\)}
        \State \texttt{Compute} \(c_j \equiv f^{p^{\omega_{j,\varepsilon-1}}\cdot\order(\bar{f},\Theta_j)} \cdot \widetilde{g}_j^{p^{\varepsilon-1}} \pmod{(m, p^\varepsilon)}\), \(j=1,\dots,\eta\) \Comment{\cref{pro:g-hat}}
        \State \texttt{Compute} \(
            \omega_{j,\varepsilon} =
                \begin{cases}
                    \omega_{j,\varepsilon-1}, &\text{if } c_j \equiv \widetilde{g}_j^{p^{\varepsilon-1}} \pmod{(m, p^\varepsilon)}\\
                    \omega_{j,\varepsilon-1} + 1, &\text{else}
                \end{cases},
                ~ j=1,\dots,\eta
        \) \Comment{\cref{the:omega}}
        \State \texttt{Set} \(\set{K}_j = \{1,p^{\omega_{j,\varepsilon}}\cdot\order(\bar{f},\Theta_j)\}\), \(j=1,\dots,\eta\) \Comment{\cref{the:cset-lcm-construction-theta}, using \(\order(\bar{f},\Theta_j)\) from line 6}
        \State \texttt{Compute} \(\set{H} = \bigodot_{i=1}^r \set{K}_i = \mleft\{\lcm(a_1, \dots, a_\eta) : a_j \in \set{K}_j\mright\}\)
        \State \texttt{Compute} \(\cset(\Gamma_\varepsilon(f,m)) = \cset(\Gamma_{\varepsilon-1}(f,m)) \bigcup \set{H}\)
    \EndFor
    \State \Return \(\cset(\Gamma_e(f,m))\)
\end{algorithmic}
\caption{Computes the set of cycle lengths \(\cset(\Gamma_e(f,m))\) using polynomials \(\Theta_j\) as defined in \cref{eq:theta}.}
\label{alg:cycle-lengths-opt}
\end{algorithm}

%% file: sections/5_transient_length.tex
\section{Height of \texorpdfstring{\glspl{lfds}}{LFDSs} on Cyclic \texorpdfstring{\(\pring{R_e}\)-Modules}{Modules Over Galois Rings}}
\label{sec:height}

\input{algorithms/transient_length}
In this section, we develop a simple algorithm to compute the height of \glspl{lfds} of the form \(\sys(\qring{\pring{R_e}}{\ideal{m}}, \allowbreak\Gamma_e(f,m))\).
Our algorithm relies on \cref{pro:height-field,the:orbit-1} and is given in \cref{alg:transient-lengths}

\begin{theorem}
    \label{the:alg-2}
    \Cref{alg:transient-lengths} correctly computes the height of the trees in the functional graph of the \gls{lfds} of the form
    \(\sys(\qring{\pring{R_e}}{\ideal{m}}, \Gamma_e(f,m))\) in time
    \(O^{\sim}(n^2 \cdot e^2 \cdot M(n))\).
\end{theorem}
\begin{proof}
    The underlying idea is simple.
    First, we compute the height of the homomorphic system over the residue field \(R_1\isomorphic\field_q\), utilizing \cref{pro:height-field}.
    Next, we apply \cref{the:lindenberg,the:orbit-1} and use a binary search to determine the actual system height over the \gls{gr}.
    Let \(h\) denote the current system height during binary search and let \(\ell_{\max}\) represent the system's maximum cycle length.
    We have that by \cref{the:orbit-1}, \(\ell \mid \ell_{\max}\) for any \(\ell \in \cset(\Gamma_e(f,m))\).
    In each step of the binary search, we examine whether the following congruence holds:
    \begin{equation*}
        f^{h+\ell_{\max}} \equiv f^{h} \pmod{(m,p^e)}.
    \end{equation*}
    If the congruence is satisfied, \(h\) is greater than or equal to the actual system height, and we adjust the binary search towards a smaller \(h\).
    If the congruence fails, \(h\) is strictly less than the actual system height, so we evaluate a larger \(h\).
    
    We can analyze the time complexity of \cref{alg:transient-lengths} in two parts.
    Lines 1--4 are dominated by the factorization of \(\bar{f}(x)\), with complexity \(O(n \cdot M(n) \cdot \log(q\cdot n))\).
    Lines 5--12 encompass the binary search.
    Since in the worst case, all states but one form a transient, and there are \(\lvert \qring{\pring{R_e}}{\ideal{m}} \rvert = q^{en}\) possible states,
    the number of iterations of the while-loop (binary search) is bounded by \(O(ne \cdot \log(q))\).
    Similar to the maximum transient length, \(\ell_{\max} < q^{en}\) such that the computation in line 7, dominated by modular exponentiation,
    operates in \(O(ne \cdot M(n) \cdot \log(q))\).
    Therefore, \cref{alg:transient-lengths} exhibits a complexity of \(O(n^2e^2 \cdot M(n) \cdot \log(q)^2)\), or \(O^{\sim}(n^2e^2 \cdot M(n))\).
\end{proof}

%% file: algorithms/transient_length.tex
\begin{algorithm}[!htb]
\begin{algorithmic}[1]
    \Require Polynomials \(f,m\in\pring{R}\), \(f\) monic, \(n=\deg(m)\), \(p\) prime, \(e\in\intnum^+\), and \(\ell_{\max} = \max(\cset(\Gamma_e(f,m)))\)
    \Ensure Height of \(\sys(\qring{\pring{R_e}}{\ideal{m}}, \Gamma_e(f, m))\)
    \State{Factor} \(\bar{f}(x) = G_1(x)^{k_1} \cdots G_r(x)^{k_s}\)
    \State{Compute} \(\tau = \max\{\lceil \kappa_{G_i}(M_1) / \kappa_{G_i}(F) \rceil : G_i \mid F\}\) \Comment{\cref{pro:height-field}}
    \State{Compute} \(r = e\tau\) \Comment{\cref{the:lindenberg}}
    \State{Initialize} \(l = \tau, h=\texttt{none}\)
    \While{\(l \le r\)} \Comment{binary search}
        \State{Compute} \(h = l + \lfloor (r-l)/2 \rfloor\)
        \If{\(f_e^{h + \ell_{\max}} \equiv f_e^{h} \pmod{m}\)} \Comment{\cref{the:orbit-1}}
        \State{Compute} \(r = h-1\)
        \Else
        \State{Compute} \(l = h+1\)
        \EndIf
    \EndWhile
    \If{\(f_e^{h + \ell_{\max}} \equiv f^{h} \pmod{m}\)}
        \State\Return \(h\)
    \Else
        \State\Return \(h+1\)
    \EndIf
\end{algorithmic}
\caption{Computes the height of \(\sys(\qring{\pring{R_e}}{\ideal{m}}, \Gamma_e(f,m))\).}
\label{alg:transient-lengths}
\end{algorithm}

%% file: sections/6_conclusion.tex
\section{Conclusion}
\label{sec:conclusion}
In this paper, we analyzed the dynamics of \glspl{lfds} that act on cyclic \(\pring{R_e}\)-modules.
Based on the \(f\)-decomposition of the defining ideal and the corresponding decomposition into nilpotent and bijective subsystems,
we developed a method to construct the set of cycle lengths.
An optimization of the computation of \(\cset(\Gamma_e(f,m))\) is achieved by grouping irreducible divisors of the defining polynomial of the bijective subsystem over \(\field_q\) via an equivalence relation
on their polynomial order.
Furthermore, we developed and analyzed practical and applicable algorithms to compute the set of cycle lengths and the system height.
With runtime complexities of \(O^{\sim}\bigl(n^{1+\log(\log(n))} + n^2e^2 + (n^2+e^2 + ne)\cdot\mathcal{M}_e(n)\bigr)\) for the cycle length analysis and 
\(O^{\sim}(n^2 \cdot e^2 \cdot M(n))\) for the system height analysis, the proposed algorithms enable the practical analysis of the dynamics of such types of \glspl{lfds}.